\documentclass[oneside]{amsart}

\usepackage[letterpaper,body={14.6cm, 22.5cm}, mag=1000]{geometry}
\usepackage{amssymb}
\usepackage{amsthm}
\usepackage{amscd}
\usepackage{lineno}

\numberwithin{equation}{section}
\theoremstyle{plain}
\newtheorem{thm}{Theorem}[section]
\newtheorem{cor}[thm]{Corollary}
\newtheorem{lemma}[thm]{Lemma}

\newtheorem*{thma}{Theorem A}
\newtheorem*{thmb}{Theorem B}

\theoremstyle{definition}

\newtheorem{remark}[thm]{Remark}

\newcommand{\dlabel}[1]{\ifmmode \text{\ttfamily \upshape [#1] } \else
{\ttfamily \upshape [#1] }\fi \label{#1} }

\newcommand{\C}{\operatorname{C} }

\newcommand{\K}{\operatorname{K} }

\newcommand{\Z}{\operatorname{Z} }

\newcommand{\gen}[1]{\left < #1 \right >}

\begin{document}
\baselineskip 15pt

\title{Commutators and commutator subgroups of finite $p$-groups}

\author{Rahul Kaushik}

\address{School of Mathematics, Harish-Chandra Research Institute, HBNI,
Chhatnag Road, Jhunsi, Allahabad - 211 019, INDIA}

\email{rahulkaushik@hri.res.in}

\author{Manoj K.~Yadav}

\address{School of Mathematics, Harish-Chandra Research Institute, HBNI, 
Chhatnag Road, Jhunsi, Allahabad - 211 019, INDIA}

\email{myadav@hri.res.in}

\subjclass[2010]{Primary 20D15, 20F12}
\keywords{commutator, commutator subgroup, $p$-group}

\begin{abstract}
We present a classification of finite $p$-groups $G$  with  $\gamma_2(G)$, the commutator subgroup of $G$, of order  $p^4$ and exponent $p$ such that not all elements of $\gamma_2(G)$ are commutators. 
\end{abstract}
\maketitle

\section{Introduction}

This paper is devoted to the study of finite $p$-groups $G$  with $\gamma_2(G)$ of order $p^4$ and exponent $p$ such that $\K(G) = \gamma_2(G)$, where $\gamma_2(G)$ denotes the commutator subgroup of $G$ and $\K(G) := \{[x, y] \mid x, y \in G\}$. A similar study for finite $p$-groups $G$ with $\gamma_2(G)$ minimally generated by at most $3$ elements has already been carried out in the past. Rodney  \cite{dmR73} proved that $\K(G) = \gamma_2(G)$ for a nilpotent group $G$ with cyclic commutator subgroup. He \cite{dmR77} also proved that $\K(G) = \gamma_2(G)$ for all finite groups $G$ with $\gamma_2(G)$ elementary abelian of order $p^3$ for any prime integer $p$. Guralnick \cite{rmG82} considered finite groups $G$ such that $\gamma_2(G)$ is an abelian $p$-group minimally generated by at most $3$ elements, and proved that $\K(G) = \gamma_2(G)$ for such groups, where $p \ge 5$. In a series of two papers Fernandez-Alcober and Heras \cite{FAH19}, and Heras \cite{iH20}, respectively,  it is proved that $\K(G) = \gamma_2(G)$ for a finite $p$-group $G$ such that $\gamma_2(G)$ is generated by $2$ elements, and for a finite $p$-group $G$, $p \ge 5$, such that $\gamma_2(G)$ is generated by $3$ elements.  A generalization of some of these results for  commutators  of higher weights has also been studied in the literature. The interested reader may refer to \cite{DN78, rmG82a, rmG83, HM20}.

Examples of groups $G$ of order $p^8$ and nilpotency class $2$ such that $\K(G) \neq \gamma_2(G)$ were constructed by Macdonald \cite[Exercise 5, Page 78]{idM68}. Such examples of groups of order $p^6$, $p \ge 5$, having nilpotency class $4$ and groups of order $2^7$ having nilpotency class $3$ were constructed by Kappe and Morse in  \cite{KM05}, where the authors proved that $\K(G) = \gamma_2(G)$ for all $p$-groups of order at most $p^5$ and for all $2$-groups of order at most $2^6$.  Some more details on the topic can be found in \cite{KM07}. So it is evident that a condition only on the prime $p$ for a finite $p$-group $G$ can not ensure that $\K(G) = \gamma_2(G)$.  The counter examples above of order $2^8$ and $p^6$, $p \ge 5$, admit  elementary abelian commutator subgroups of order $2^4$ and $p^4$ respectively. This motivates us to look at the classification of finite $p$-groups $G$  with $\gamma_2(G)$ elementary abelian  of rank $4$ such that $\K(G) \ne \gamma_2(G)$. This study will also reveal patterns for future investigation on the topic. It follows from Lemma \ref{prelemma} that the property under consideration is invariant under isoclinism of groups (see Section 2 for the definition), which is an equivalence relation. As observed in the next section, each isoclinism family admits a group $G$ such that $\Z(G) \le \gamma_2(G)$, where $\Z(G)$ denotes the center of $G$. In the following result we  provide the desired characterization of  finite $p$-groups (upto isoclinism) whose commutator subgroups are of order $p^4$ and exponent $p \ge 3$.

\begin{thma}
Let $G$ be a  finite $p$-group with $\Z(G) \le \gamma_2(G)$ and $\gamma_2(G)$ of order $p^4$ and exponent $p \ge 3$. Then $\K(G) \ne \gamma_2(G)$ if and only if one of the following holds:

{\rm (1)} $G$ is of order $p^6$ and nilpotency class $4$ with $|\Z(G)| = p^2$.

{\rm (2)} $G$ is of order $p^7$ and nilpotency class $3$ with $|\Z(G)| = p^3$.

{\rm (3)} $G$ is of order $p^8$ and nilpotency class $2$ along with one of the following:
\begin{subequations}
\begin{align}
&\text{\hspace{.05in} {\rm (3a)} $G$ admits a non-central element whose centralizer in $G$ is a maximal subgroup.}\nonumber\\
&\text{ \hspace{.0in} {\rm (3b)} $G$ is of conjugate type $\{1, p^2, p^3\}$ and admits no  generating set $\{x_1, x_2, x_3, x_4\}$}\nonumber \\
&\text{\hspace{.35in} such that $[x_1, x_2] = 1 = [x_3, x_4]$.}\nonumber
\end{align}
\end{subequations}
Moreover, if $\K(G) \ne \gamma_2(G)$, then every element of $\gamma_2(G)$ can be written as a product of at most two elements from $\K(G)$.
\end{thma}

Along with the known results in the literature (as mentioned above), the preceding theorem provides a characterization of finite $p$-groups, $p \ge 5$, with   $\gamma_2(G)$ of order  $p^4$ such that each element of $\gamma_2(G)$ is a commutator. A classification of $p$-groups of order $p^6$, $p \ge 3$, with $\K(G) = \gamma_2(G)$ is evident. 

We now turn our attention to $2$-groups. It is known that $\K(G) = \gamma_2(G)$ for all $2$-groups $G$ of order at most $2^6$  \cite{KM05}. The following theorem provides a characterization (upto isoclinism) of finite $2$-groups $G$ with $\gamma_2(G)$  elementary abelian $2$-group of order $16$ such that $\K(G) \ne \gamma_2(G)$.

\begin{thmb}
Let $G$ be a finite $2$-group such that $\Z(G) \le \gamma_2(G))$ and $\gamma_2(G)$ is elementary abelian of order $16$. Then $\K(G) \ne \gamma_2(G)$  if and only if one of the following holds true:

{\rm (1)}  $G$ is isoclinic to the following special $2$-group of order  $2^{9}$ presented as: 
\begin{eqnarray*}
& & \langle  v_1, v_2, v_3, v_4, v_5 \mid [v_1, v_2] =  [v_2, v_3] =  [v_3, v_1] = [v_4, v_2] = [v_5, v_1] =  1, [v_3, v_4]=[v_3, v_5],\\
& & [x, y, z] = 1 \mbox{ for all } x, y, z \in \{v_1, \ldots, v_5\}, v_i^2 = 1  \;( 1\le i \le 5)\rangle.
\end{eqnarray*}

{\rm (2)} $G$ is of order $p^8$ and nilpotency class $2$ along with one of the following:
\begin{subequations}
\begin{align}
&\text{\hspace{.05in} {\rm (2a)} $G$ admits a non-central element whose centralizer in $G$ is a maximal subgroup.}\nonumber\\
&\text{ \hspace{.0in} {\rm (2b)} $G$ is of conjugate type $\{1, 4, 8\}$ and admits no  generating set $\{x_1, x_2, x_3, x_4\}$}\nonumber \\
&\text{\hspace{.35in} such that $[x_1, x_2] = 1 = [x_3, x_4]$.}\nonumber
\end{align}
\end{subequations}
Moreover, if $\K(G) \ne \gamma_2(G)$, then every element of $\gamma_2(G)$ can be written as a product of at most two elements from $\K(G)$.
\end{thmb}

A GAP computation shows that there are $2917$ groups $G$  of order $2^8$ with $\gamma_2(G)$ elementary abelian of order $16$, out of which  $1542$ groups  satisfy $\K(G) = \gamma_2(G)$.

Our notations for objects associated with a finite multiplicative group $G$ are mostly standard.   If $x,y \in G$, then $x^y$ denotes the conjugate element 
$y^{-1}xy \in G$ and $[x,y] = [x,y]_G$ denotes the commutator  $x^{-1}y^{-1}xy = x^{-1}x^y \in G$. If $x \in G$, then  $[x,G]$  denotes the set  
$\{[x, g] \mid g \in G\}$ (and not the subgroup generated by these commutators). If $[x,G] \subseteq \Z(G)$, then $[x,G]$ becomes a subgroup of $G$. For $n \ge 1$, we define $[x, {}_{n+1}y] = [[x, {}_ny], y]$, where $[x, {}_1y] = [x, y]$.
For a subgroup $H$ of $G$, $\C_{G}(H)$ denotes the centralizer of $H$ in $G$ and for an element $x \in G$, $\C_{G}(x)$ denotes the centralizer of $x$ in $G$. 
By $\mathbb{F}_p$ we denote the field of $p$ elements. We write the subgroups in the lower central series of $G$ as $\gamma_n(G)$,  where $n$ runs over all strictly positive integers. For a finite $p$-group $G$, we define $\Omega_i(G) := \langle x \in G \mid x^{p^i} = 1\rangle$ and  $\mho_i(G) = \langle x^{p^i} \mid x \in G \rangle$. 


In Section 2 we present some preliminaries and a reduction argument. Sections 4, 5 and 6 deal with $p$-groups under consideration having nilpotency class 2, 4 and 3, respectively.  Proof of Theorem A is presented in Section 6. $2$-groups under consideration are dealt with in Section 7, where a proof of Theorem B is presented. Several examples are exhibited in Section 8. We conclude this section with the remark that we  used Magma \cite{BCP} and GAP \cite{GAP} for establishing our results for small  primes, before writing final proofs.


\section{Reductions}

We start with the following concept of isoclinism of groups,  introduced by P. Hall \cite{pH40}.

Let $X$ be a  group and $\bar{X} = X/\Z(X)$. 
Then commutation in $X$ gives a well defined map
$a_{X} : \bar{X} \times \bar{X} \mapsto \gamma_{2}(X)$ such that
$a_{X}(x\Z(X), y\Z(X)) = [x,y]$ for $(x,y) \in X \times X$.
Two  groups $G$ and $H$ are called \emph{isoclinic} if 
there exists an  isomorphism $\phi$ of the factor group
$\bar G = G/\Z(G)$ onto $\bar{H} = H/\Z(H)$, and an isomorphism $\theta$ of
the subgroup $\gamma_{2}(G)$ onto  $\gamma_{2}(H)$
such that the following diagram is commutative
\[
 \begin{CD}
   \bar G \times \bar G  @>a_G>> \gamma_{2}(G)\\
   @V{\phi\times\phi}VV        @VV{\theta}V\\
   \bar H \times \bar H @>a_H>> \gamma_{2}(H).
  \end{CD}
\]
The resulting pair $(\phi, \theta)$ is called an \emph{isoclinism} of $G$ 
onto $H$. Notice that isoclinism is an equivalence relation among  groups.  

Let $G$ be a finite $p$-group. Then it follows from \cite{pH40} that there exists a finite $p$-group $H$ in the isoclinism family of $G$ such that 
$\Z(H) \le \gamma_2(H)$. Such a group $H$ is called a \emph{stem group} in the isoclinism family of $G$.  As an easy consequence of this statement, we get

\begin{cor}\label{precor}
A finite $p$-group $G$ with $|\gamma_2(G)| = p$ is isoclinic to an extraspecial $p$-group.
\end{cor}

 \begin{lemma}\label{prelemma}
Let $G$ and $H$ be two isoclinic finite $p$-groups.  Then  $\K(G) = \gamma_2(G)$ if and only if $\K(H) = \gamma_2(H)$.
\end{lemma}
\begin{proof}
Since isoclinism is an equivalence relation, it is sufficient to prove one side implication. So assume that $\K(G) = \gamma_2(G)$, and  the map $a_G$ in the above commutative diagram is surjective. Let $u \in \gamma_2(H)$ be an arbitrary element. There exists $(\bar{g}_1, \bar{g}_2) \in \bar G \times \bar G$ such that $\theta(a_G(\bar{g}_1, \bar{g}_2)) = u$. Thus there exists $(\bar{h}_1, \bar{h}_2) \in  \bar H \times \bar H$, namely $(\phi(\bar{g}_1), \phi(\bar{g}_2))$, such that $a_H(\bar{h}_1, \bar{h}_2) = u$ proving that $a_H$ is surjective, and the proof is complete.  \hfill $\Box$

\end{proof}

In view of the preceding lemma, it is sufficient to consider a stem group from a given isoclinism family.

\begin{lemma}\label{prelemma1}
Let $G$ be a  finite $p$-group such that $\gamma_2(G)$ is elementary abelian of order $p^4$. If the nilpotency class of $G$ is at most $3$, then $\{x^p \mid x \in G\} \subseteq \Z(G)$ for all  $p \ge 3$. Otherwise the conclusion, in general, holds only for $p \ge 5$.
\end{lemma}
\begin{proof}
 Since $\gamma_2(G)$ is of order $p^4$, the nilpotency class of $G$ is at most $5$. Thus, $\gamma_2(G)$ being elementary abelian,  for all $x, y \in G$, we have
 $$ [x,y^p]=[x,y]^p [x,y,y] ^{\binom p2} [x,y,y,y]^{\binom p3}  [x,y,y,y,y]^{\binom p4}=1. $$
 Hence $y^p \in Z(G)$.
  \hfill $\Box$

\end{proof}

\begin{lemma}\label{prelemma2}
For a  finite $p$-group $G$ of nilpotency class at least $4$, $\Z(G) \cap \gamma_2(G)$ can not be maximal in $\gamma_2(G)$.
\end{lemma}
\begin{proof}
Contrarily assume that $\Z(G) \cap \gamma_2(G)$ is maximal in $\gamma_2(G)$. Now consider the quotient group $\bar{G} := G/\gamma_4(G)$, whose nilpotency class is $3$. Thus  $\gamma_2(\bar{G}) \not\le \Z(\bar{G})$. This means that $\Z(\bar{G})$ can not grow in   $\gamma_2(\bar{G}) = \gamma_2(G)/\gamma_4(G)$. $\bar{G}$ being of nilpotency class $3$, we have $\gamma_3(\bar{G}) \le \Z(\bar{G}) \cap \gamma_2(\bar{G}) =  \big(\Z(G) \gamma_4(G)/\gamma_4(G)\big) \cap \gamma_2(G)/\gamma_4(G)$, which is possible only when 
$\gamma_3(G) \le \Z(G)$, a contradiction. \hfill $\Box$

\end{proof}

The proof of each of the following three lemmas is straightforward.
\begin{lemma} \label{prelemma3}
Let $G$ be a finite group and $H\leq \gamma_2(G) \cap Z(G)$. If there exist $x_1,x_2, \ldots , x_n$ such that $\gamma_2(G)/H= \bigcup\limits_{i=1}^{n} [x_iH,G/H] $ and $H\subseteq \bigcap \limits_{i=1}^{n} [x_i, G]$,  then $\gamma_2(G)= \bigcup \limits_{i=1}^{n} [x_i, G]$.
\end {lemma}

\begin{lemma} \label{prelemma4}
Let a finite group $G$ be a central product of its subgroups $M$ and $N$. Then $\K(G) = \gamma_2(G)$ if and only if $\K(M) = \gamma_2(M)$ and $\K(N) = \gamma_2(N)$.
\end{lemma}

\begin{lemma}\label{prelemma5}
Let $G$ be a finite $p$-group and $H$ a normal subgroup of $G$ of order $p$ contained in $\gamma_2(G)$ such that $\gamma_2(G/H) =\K(G/H)$. Then every element of $\gamma_2(G)$ can be written as a product of at most two elements from $\K(G)$.
\end{lemma}

Following  \cite[p. 28]{gorenstein}, a group $G$ is said to be an {\it amalgamated (internal) semidirect product} of subgroups $H$ by $K$ over $M$, written $G=H\rtimes_M K$,  if $H \trianglelefteq G$, $G=HK$ and $H\cap K=M$. In particular, if $[H,K]=1$ then $M\le Z(G)$, and we call $G$ the {\it central product} of $H$ by $K$ over $M$, written $G=H\times_M K$.

For a finite $p$-group $G$, the {\it breadth} of an element $x \in G$, denoted by $b(x)$, is defined as 
\[p^{b(x)} := |G:\C_G(x)|,\]
and the {\it breadth} of $G$, denoted by $b(G)$, is defined as 
\[b(G) := max\{b(x) \mid x \in G\}.\]

\begin{remark}\label{remark}
Let $G$ be a finite $p$-group with $|\gamma_2(G)| = p^4$. Then by \cite[Theorem A]{PS99} it follows that $b(G) \ge 3$. We'll use this information throughout without any further reference. 
\end{remark}

We'll need the following result of Rodney \cite[Theorem B]{dmR77}.

\begin{thm}\label{dmr77}
Let $G$ be a finite group such that $\gamma_2(G)$  is elementary abelian of order $p^3$, where $p$ is any prime integer. Then $\K(G) = \gamma_2(G)$.
\end{thm}

Our reduction arguments heavily depend on the following result of Parmeggiani and Stellmacher \cite[Corollary]{PS99}:

\begin{thm}\label{psthm}
Let $p$ be an odd prime and $G$  a finite $p$-group. Then $b(G) = 3$ if and only if one of the following holds:

{\rm (i)} $|\gamma_2(G)| = p^3$ and $|G :  \Z(G)| \ge p^4$.

{\rm (ii)} $|G :  \Z(G)| = p^4$ and $|\gamma_2(G)| \ge p^4$.

{\rm (iii)} $|\gamma_2(G)| = p^4$ and there exists a normal subgroup $H$ of $G$ with $|H| = p$  and $|G/H : \Z(G/H)| = p^3$. 
\end{thm}

We remark that the preceding result was also proved in \cite[Corollary 3]{GMMPS} for $p \ge 5$.

The following is the key result which reduces our study mainly to the groups of small orders.
\begin{thm}\label{keyresult}
Let $L$ be a finite $p$-group  such that $\Z(L) \le \gamma_2(L)$, $\gamma_2(L)$ is elementary abelian of order $p^4$ and $b(L) = 3$.   If the nilpotency class of $L$ is $3$ and $p \ge 3$, then one of the following  holds:

{\rm (i)} There exists a $2$-generator subgroup $G$ of $L$ having the same nilpotency class as that of $L$ such that $\gamma_2(G) = \gamma_2(L)$. Moreover, $|G| = p^6$,  and if $|L| \ge p^7$, then $L$ is an amalgamated semidirect product of $G$ and a subgroup $K$ with $|\gamma_2(K)| \le p$.  Moreover, if  $K$ is non-abelian, then it is isoclinic to an extraspecial $p$-group.

{\rm (ii)} There exists a $2$-generator subgroup $G$ of $L$ having the same nilpotency class as that of $L$ such that $\gamma_2(G) < \gamma_2(L)$. Moreover, $|G| = p^5$, $|L| \ge p^7$ and $L$ is a central  product of $G$ and a subgroup $K$ of nilpotency class $2$, which is isoclinic to an extraspecial $p$-group.

{\rm (iii)} There exists a $3$-generator subgroup $G$ of $L$ having the same nilpotency class as that of $L$ such that $\gamma_2(G) = \gamma_2(L)$. Moreover, $|G| = p^7$,  and if $|L| \ge p^8$, then $L$ is an amalgamated semidirect product of $G$ and a subgroup $K$ with $|\gamma_2(K)| \le p$.  Moreover, if  $K$ is non-abelian, then it is isoclinic to an extraspecial $p$-group.

If  the nilpotency class of $L$ is $4$ and $p \ge 3$, then  only {\rm (i)}  holds.
 \end{thm}
\begin{proof}
Since $b(L) = 3$, it follows from Theorem \ref{psthm} that either $|L:\Z(L)| = p^4$ or $L$ admits a subgroup $H$ of order $p$ such that $|L/H:\Z(L/H)|=p^3$.
If $|L:\Z(L)| = p^4$, then it follows from Lemma \ref{prelemma1} that $L$ itself is either a $2$-generator group of order $p^6$ or a $3$-generator group of order $p^7$. For, $p^2 \le |\Z(L)| \le p^3$; otherwise the nilpotency class of $L$  will at most be $2$, which we are not considering. Moreover, when $L$ is a $3$-generator group of order $p^7$, then it follows from Lemma \ref{prelemma2} that the nilpotency class of $L$ is $3$. Now consider the second case, which we divide into two subcases, depending on the nilpotency class of $L$. 

First assume that the nilpotency class  of $L$ is $3$. Then $H = \gamma_3(L)$, $H < \gamma_3(L)$ or  $H \not\le \gamma_3(L)$. We consider these possibilities one by one. If $H = \gamma_3(L)$, then the nilpotency class of $L/H$ is $2$ and $|L/H:\Z(L/H)|=p^3$, and therefore, using Lemma \ref{prelemma1}, it follows that except three generators $a$, $b$, $c$ (say) of $L$, all other generators $x_1, x_2, \ldots, x_k$, $k \ge 0$, are such that $[x_i, L] \le H$. Since $\Z(L) \le \gamma_2(L)$, we can, more precisely, say that $[x_i, L] = H$ for all $1 \le i \le k$. So it follows that $\gamma_2(L)/H = \gamma_2(G)H/H$ is of order $p^3$, where $G:=\gen{a, b, c}$ is a subgroup of $L$. We claim that $H \le \gamma_2(G)$. As observed above  $[x_i, L] = H$ is of order $p$, it follows that, for all $ 1 \le i \le k$, $\C_L(x_i)$ is maximal in $L$, and therefore contains $\gamma_2(L)$. Thus any generator $h$ of $H$, which lies in $\gamma_3(L) = [\gamma_2(L), L]$,  can be written as 
\[h = [w_1^{\alpha_1}w_2^{\alpha_2}w_3^{\alpha_3}, a^{\beta_1}b^{\beta_2}c^{\beta_3}] \in \gamma_3(G),\]
where $w_1H, w_2H, w_3H$ generate $\gamma_2(G)H/H$. The nice presentation of $h$ in the preceding statement is possible because $H \le \Z(L)$. Hence our claim follows, which, in turn, implies that $\gamma_2(G) = \gamma_2(L)$. 
Thus $G$ is of order $p^7$ and nilpotency class $3$. Let $K := \gen{x_1, \ldots, x_k}$ be a subgroup of $L$. Since $[x_i, L] = H$ for $1 \le i \le k$, we have $\gamma_2(K) \le H$, which shows that the nilpotency class of $K$ is at most $2$.  Since $\gamma_2(L) \le G$, $K$ acts on $G$ by conjugation. Hence  $L$ takes the desired form. If $K$ is non-abelian, then, in view of Corollary \ref{precor}, $K$ is  isoclinic to an extraspecial $p$-group.

If $H <  \gamma_3(L)$, then the nilpotency class of $L/H$ is  $3$ and $|L/H : \Z(L/H)| = p^3$.  Hence by the given hypothesis and Lemma \ref{prelemma1}, we conclude that $L$ can be generated by $\{a, b, x_1, \ldots, x_k\}$ such that $[x_i, L] = H$. Now, using the same arguments as in the preceding case, the assertion follows  by assuming $G := \gen{a, b}$ and $K := \gen{x_1, \ldots, x_k}$, where $|G|= p^6$. Finally, if $H \not\le \gamma_3(L)$, then, obviously,  $H \cap \gamma_3(L) = 1$. Thus the nilpotency class of $L/H$ is also $3$. Again invoking the given hypothesis and Lemma \ref{prelemma1}, we can assume that $L$ is generated by the set $\{a, b, x_1, \ldots, x_k\}$ such that $[x_i, L] = H$. Let $G_1 :=  \gen{a, b}$. Notice that  $\gamma_2(G_1)H/H = \gamma_2(L)/H$ is of order $p^3$, $|G_1H/H| = p^5$, and $G_1$ and $L$ agree on the nilpotency class.  Since $G_1$ is  $2$-generator, $\gamma_2(G_1)/\gamma_3(G_1)$ is cyclic (of order $p$). Thus $G_1$ can not contain $H$, which implies that $|G_1| = p^5$. If $G_1 \le \C_L(x_i)$, for all $1 \le i \le k$, then $K:=\gen{x_1, \ldots, x_k}$ with  $\gamma_2(K) = H$  is isoclinic to an extraspecial $p$-group, and therefore $L$ is a central product of $G_1$ and $K$ amalgamating some subgroup (possibly trivial). Hence $G=G_1$ and $K$ are the desired subgroups. So assume that $[x_i, G_1] \neq 1$ for some $i$. Thus $[x_i, G_1] = H$, and the subgroup $G:=\gen{a, b, x_i}$ of $L$ is of order $p^7$. Hence, as argued above, one can easily see  that $G$ and $K:=\gen{x_1, \ldots, x_{i-1}, x_{i+1}, \ldots, x_k}$ are the desired subgroups of $L$.

We now assume that the nilpotency class of $L$ is $4$. Then either $H = \gamma_4(L)$ or  $H \ne \gamma_4(L)$. 
We claim that there is no $L$ of nilptency class $4$ such that  $H \ne \gamma_4(L)$. If such an $L$ exists, then  the nilpotency class of $L/H$ is  $4$, which is not possible as $(L/H)/\Z(L/H)$, being of order $p^3$,  can have nilpotency class at most $2$.  So assume that $H = \gamma_4(L)$. Then the nilpotency class of $L/H$ is  $3$.  Hence, by the given hypothesis, we conclude that $L$ can be generated by $\{a, b, x_1, \ldots, x_k\}$ such that $[x_i, L] = H$. Since $|L/H : \Z(L/H)| = p^3$, it follows that $a^p, b^p \in \gamma_2(L)$. Now, using the same arguments as above, the assertion follows  by assuming $G := \gen{a, b}$ and $K := \gen{x_1, \ldots, x_k}$, where $|G|= p^6$. This completes the proof of the theorem.   \hfill $\Box$

\end{proof}


\section{groups of class $2$}\label{sec3}

This section is devoted to the investigation of the question under consideration for  groups of class $2$. We say that a finite $p$-group $G$ is of conjugate type  $\{1, p^r, p^s\}$ if this set constitutes the set of  conjugacy class sizes of all  elements of $G$, where $r \le s$ are positive integers.

\begin{lemma} \label{cl2lemma1} 
 Let $G$ be a finite $p$-group of order $p^8$ and nilpotency class $2$  such that $\gamma_2(G)$ is elementary abelian of order $p^4$. If $G$ is not of conjugate type $\{1, p^3\}$ or $\{1, p^2, p^3\}$, then $\K(G) \neq G'$. Moreover, every element of $\gamma_2(G)$ can be written as a product of at most two elements from $\K(G)$.
\end{lemma}
\begin{proof} It follows by the given hypothesis that  $\Z(G) = \gamma_2(G)$ and $G$ is minimally generated by $4$ elements.  Notice that $b(G)=3$.  Again by the given hypothesis there exists an element $d \in G - \gamma_2(G)$ such that  $\C_G(d)$ is maximal in $G$. We can always extend $\{d\}$ to a generating set $\{a, b, c, d\}$ for $G$ such that $\gamma_2(G)=\langle [a,b], [a,c], [b,c], [c,d]\rangle$.  We claim that $[a,b][c,d]$ is not in $\K(G)$. Contrarily assume that $[a,b][c,d] \in \K(G)$. Thus 
$$[a, b][c, d] = [a^{\alpha_1} b^{\alpha_2} c^{\alpha_3} d^{\alpha_4 } , a^{\beta_1} b^{\beta_2} c^{\beta_3} d^{\beta_4}],$$
 where $\alpha_i, \beta_j \in \mathbb{F}_p$ for $1 \le i, j \le 4$. Expanding the right hand side and comparing the powers of the generators of $\gamma_2(G)$, we get the following set of equations:
\begin{eqnarray}
\beta_2 \alpha_1 - \alpha_2 \beta_1 & = & 1, \label{eqn1} \\     
\beta_3 \alpha_2 - \alpha_3 \beta_2 &=& 0, \label{eqn2}\\
\beta_3 \alpha_1 - \alpha_3 \beta_1 &=& 0, \label{eqn3}\\
\beta_4 \alpha_3 - \alpha_4 \beta_3 &=&1.\label{eqn4}
\end{eqnarray}

First assume that  $\beta_3 \neq 0$. Then from \eqref{eqn2} and \eqref{eqn3}, we get $\alpha_2 = \alpha_3\beta_2{\beta_3}^{-1} $ and $\alpha_1 = \alpha_3\beta_1{\beta_3}^{-1}$. Notice that these values of $\alpha_1$ and $\alpha_2$ contradict \eqref{eqn1}. Thus  $\beta_3=0$. That $\alpha_3 \neq 0$ follows from \eqref{eqn4} after inserting  $\beta_3 = 0$. This then implies, along with  \eqref{eqn2} and \eqref{eqn3}, that $\beta_1=\beta_2=0$, which contradicts \eqref{eqn1}. Hence the above system of equations has no solution, which settles our claim. 

Let $H$ be any subgroup of $\gamma_2(G)$ of order $p$. Then it follows from Theorem \ref{dmr77} that $\gamma_2(G/H) = \K(G/H)$. Hence the  proof is complete by Lemma \ref{prelemma5}.  \hfill $\Box$

\end{proof}

\begin{lemma}  \label{cl2lemma1a}
Let $G$ be a finite $p$-group of order $p^8$,  nilpotency class $2$  and conjugate type $\{1, p^2, p^3\}$ such that $\gamma_2(G)$ is elementary abelian of order $p^4$. Then $\K(G) = \gamma_2(G)$ if and only if $G$ admits a generating set $\{a, b, c, d\}$ such that $[a, b] = 1 = [c, d]$. Moreover, if $\K(G) \ne \gamma_2(G)$, then every element of $\gamma_2(G)$ can be written as a product of at most two elements from $\K(G)$.
\end{lemma}
\begin{proof}
Let $G$ admit a generating set $\{a, b, c, d\}$ such that $[a, b] = 1 = [c, d]$. Then
 $\gamma_2(G)=\langle [a, c], [b, c], [a, d], [b,d] \rangle$.
Given $\varepsilon, \lambda, \mu, \eta \in \mathbb{F}_p$, we are going to show that 
\begin{equation}\label{cls2eqn0}
[a, c]^{\varepsilon}[b, c]^{ \lambda } [a, d]^{\mu} [b,
d]^{\eta} = [a^{\alpha_1} b^{\alpha_2} c, a^{\beta_1}  b^{\beta_2}
d^{\beta_4}]
\end{equation}
for some  $\alpha_i,  \beta_j \in \mathbb{F}_p$. We go in the reverse way. Expanding the right hand side of \eqref{cls2eqn0} and comparing powers on both sides, we get
$$ \beta_1   =  -\varepsilon, \; \beta_2  = -\lambda, \; \alpha_1 \beta_4   = \mu, \; \alpha_2 \beta_4  = \eta.$$
It is easy to see that for given $\varepsilon, \lambda, \mu, \eta \in \mathbb{F}_p$, we can always find some $\alpha_i, \beta_j \in \mathbb{F}_p$ such that \eqref{cls2eqn0} holds true. Hence $\K(G) = \gamma_2(G)$.

We provide a contrapositive proof of the converse part. We assume that $G$ admits no generating set $\{x_1, x_2, x_3, x_4\}$ such that  $[x_1, x_2] = 1 = [x_3, x_4]$.  By the given hypothesis, we can always choose a generating set $\{a, b, c, d\}$ for $G$ such that  $[a, c] = 1$ and none of the other basic commutators of weight two in generators is trivial. So we can assume that
\begin{equation*}\label{cls2eqn1}
[c, d]= [a, b]^{t_1} [b, c]^{t_2} [b, d]^{t_3} [a,d]^{t_4},
\end{equation*}
 where $t_i \in \mathbb{F}_p$ and $1 \le i \le 4$, and therefore
 $\gamma_2(G)= \langle [a, b], [b, c], [b, d], [a,d] \rangle$. 

We can write the preceding equation as
\begin{equation}\label{cls2eqn1}
[b, a^{-t_1} c^{t_2} d^{t_3}][d, a^{-t_4}c]= 1.
\end{equation}
We claim that $t_3 = 0$.  If $t_3 \neq 0$, then replacing $d$ by $d' := a^{-t_1} c^{t_2}
d^{t_3}$ and a proper substitution reduces \eqref{cls2eqn1}  to $[d', b^{-1}
a^{-t_4 t_3^{-1}} c^{t_3^{-1}}]=1$. Now replacing $b$ by  $b' := b^{-1} a^{-t_4 t_3^{-1}} c^{t_3^{-1}}$, we
get a  generating set $\{ a, b',c,d' \}$ for $G$ such that $\gamma_2(G)= \langle
[a, b'], [b', c], [b', d'], [c,d'] \rangle$ and $[a, c]=1=[b', d']$, which contradicts our hypothesis. Our claim is now settled.

So now onwards we assume that  $t_3=0$.  Hence  \eqref{cls2eqn1} reduces to $[b, a^{-t_1}
c^{t_2} ][d, c a^{-t_4}]= 1$. Now replace $c$ by $c' := c a^{-t_4} $. A simple computation gives
 $$[b, a^{-t_1} c^{t_2} ] = [b, a^{-t_1} (c'a^{t_4})^{t_2} ]= [b, a^{t_4 t_2 - t_1} c'^{t_2} ].$$ 
Thus \eqref{cls2eqn1} reduces to 
\begin{equation}\label{cls2eqn1a}
[b, a^{t_4t_2 - t_1} c'^{t_2}][d, c'] = 1.
\end{equation}

We claim that $t_4 t_2 - t_1 \ne 0$.  Contrarily assume that $t_4 t_2 - t_1 = 0$. Then \eqref{cls2eqn1a} takes the form $[b,  c'^{t_2}][d, c'] = 1,$ which gives $[b^{t_2}d, c']  = 1$. Now replacing $d $ by $d' := b^{t_2}d$, we get a generating set $\{a, b, c', d' \}$ of $G$ such that $[a, c'] = 1 = [d', c']$,  which gives that the size of the conjugacy class of $c'$ in $G$ is $p$,  a contradiction to the given hypothesis. This settles our claim.
 
So we now assume  $t_4 t_2- t_1 \neq 0$. Then replacing $a$ by $a' := a^{t_4 t_2 - t_1} c'^{t_2}$, we get a new generating set $\{ a', b, c', d \}$ such that $\gamma_2(G)= \langle [a', b], [b, c'], [b, d], [a',d] \rangle$, 
$[a', c']=1$ and, by \eqref{cls2eqn1a}, $[a', b]^{-1} = [c', d]$.  Now we claim that $[b, c']^{ \lambda } [a', d]^{\mu} \notin  \K(G)$ for some $\lambda, \mu \in \mathbb{F}_p^*$. Contrarily assume that $[b, c']^{\lambda } [a', d]^{\mu} \in \K(G)$ for all $\lambda, \mu \in \mathbb{F}_p^*$. Thus
$$[b, c']^{ \lambda } [a', d]^{\mu} = [a'^{\alpha_1} b^{\alpha_2}
c'^{\alpha_3} d^{\alpha_4 } , a'^{\beta_1} b^{\beta_2} c'^{\beta_3}
d^{\beta_4}],$$
 where  $\alpha_i,  \beta_j \in
\mathbb{F}_p$ for $1 \le i, j \le 4$. Expanding the right hand side and  comparing powers both side, we get
\begin{eqnarray}
\alpha_1 \beta_4 - \alpha_4 \beta_1 &=& \mu,\label{eqn16}\\
\alpha_2 \beta_4  - \alpha_4 \beta_2 &=& 0, \label{eqn15}\\
\alpha_2 \beta_3  - \alpha_3 \beta_2   &=& \lambda, \label{eqn14}\\
\alpha_1 \beta_2  - \alpha_2 \beta_1 - \alpha_3 \beta_4  + \alpha_4
\beta_3 & = & 0. \label{eqn13} 
\end{eqnarray}

First assume that $\alpha_2 = 0$. Then from \eqref{eqn14} $\alpha_3
\beta_2 = - \lambda$, therefore by \eqref{eqn15} $\alpha_4 = 0$. Now by
\eqref{eqn16} $\alpha_1 \beta_4 = \mu$. Since $\alpha_1$ and $\alpha_3$ are non-zero, by substituting the value of
$\beta_2 = -\lambda \alpha_3^{-1}$ and $\beta_4 = \mu \alpha_1^{-1}$ in
\eqref{eqn13}, we get $ \lambda \alpha_1  \alpha_3^{-1} +  \mu \alpha_3 \alpha_1^{-1}  = 0$. Thus $ (\alpha_1 \alpha_3^{-1})^2 = - \lambda^{-1} \mu$, which is a contradiction because 
we can choose $\lambda, \mu \in \mathbb{F}_p^*$ such that $ - \lambda^{-1}
\mu$ is  non square. Now we assume that $\alpha_2 \neq 0$. If $\alpha_4 = 0$, then \eqref{eqn15} 
implies that $\beta_4 =0$, which contradicts \eqref{eqn16}. So finally assume that both $\alpha_2$ and $\alpha_4$ are non zero.  The augmented matrix of above system of equations, with $\beta$'s as variables,  is given by

\[
M=
  \begin{bmatrix}
    -\alpha_4 & 0 & 0 & \alpha_1 & \mu  \\
    0 & -\alpha_4 & 0 & \alpha_2 & 0 \\
    0 & -\alpha_3 & \alpha_2 & 0 & \lambda \\
    -\alpha_2 & \alpha_1 & \alpha_4 & -\alpha_3 & 0
  \end{bmatrix}.
\]
Performing row operations

(i) $R_1 \to -\alpha_4^{-1} R_1, \;  R_2 \to -\alpha_4^{-1} R_2$,

(ii) $R_4 \to R_4 + \alpha_2 R_1, \; R_3 \to R_3 + \alpha_3 R_2$,

(iii) $R_3 \to \alpha_2^{-1} R_3, \; R_4 \to R_4 - \alpha_1 R_2$,

(iv) $R_4 \to R_4 - \alpha_4 R_3$,\\
we get
\[
M_1=
  \begin{bmatrix}
    1 & 0 & 0 & -\alpha_1 \alpha_4^{-1} & -\mu \alpha_4^{-1} \\
    0 & 1 & 0 & - \alpha_2 \alpha_4^{-1} & 0 \\
    0 & 0 & 1 &  -\alpha_3 \alpha_4^{-1} & \lambda \alpha_2^{-1} \\
    0 & 0 & 0 & 0 & -\mu \alpha_2 \alpha_4^{-1} - \lambda \alpha_2^{-1}
\alpha_4
  \end{bmatrix}.
\]

If the above system of equations admits a solution, then  $\mu \alpha_2 \alpha_4^{-1} + \lambda \alpha_2^{-1} \alpha_4 =0$. This gives  $(\alpha_2 \alpha_4^{-1})^2  =  - \lambda \mu^{-1}$, which is a contradiction again, because we can
choose $\lambda, \mu \in \mathbb{F}_p^*$ such that $-\lambda \mu^{-1}$ is non
square.  The final assertion holds by taking $G/H$, where $H$ is any subgroup  of  $\gamma_2(G)$ of order $p$, and then  using Theorem \ref{dmr77} and Lemma \ref{prelemma5}. The proof of the lemma is now complete. \hfill $\Box$

\end{proof}

 \begin{lemma} \label{cl2lemma1b} 
 Let $G$ be a finite $p$-group of order $p^8$,  nilpotency class $2$  and conjugate type $\{1,  p^3\}$ such that $\gamma_2(G)$ is elementary abelian of order $p^4$. Then $\K(G) = \gamma_2(G)$.
 \end{lemma}
 \begin{proof} 
 Using \cite[Lemma 3.14]{NY18}, we can assume that the exponent of $G$ is $p$. Now, it follows from \cite[Theorem 1.2]{NY18} that $G$ has the following presentation:
\[G = \gen{a, b, c, d \mid a^p = b^p = c^p = d^p =1,  [c, d]= [a, b]^{-1}, [a, c] = [b,  d]^{-r}, [[x, y], z] =1}\]
 where $x, y, z \in \{a, b, c, d\}$ and $r$ is any non-quadratic residue mod $p$.
 We'll show that for given $\lambda, \mu, \nu, \xi \in \mathbb{F}_p$, there exist  $\alpha_i, \ \beta_i \in \mathbb{F}_p$ such that
 $$[a, b]^{\lambda} [b, c]^{\mu} [b, d]^{\nu} [a, d]^{\xi} = [a^{\alpha_1} b^{\alpha_2} c^{\alpha_3} , \  a^{\beta_1} b^{\beta_2} c^{\beta_3} d^{\beta_4}].$$

Solving both sides of the preceding equation and comparing the powers, we get
\begin{eqnarray}
\alpha_1 \beta_2 - \alpha_2 \beta_1 - \alpha_3 \beta_4   &=& \lambda,
\label{eqn17}\\
\alpha_2 \beta_3  - \alpha_3 \beta_2 &=& \mu , \label{eqn18}\\
\alpha_2 \beta_4 -r (\alpha_1 \beta_3  - \alpha_3 \beta_1)   &=& \nu,
\label{eqn19}\\
\alpha_1 \beta_4   & = & \xi. \label{eqn20}
\end{eqnarray}
It is sufficient to show that this system of equations, $\beta_i$'s as variables, admits a solution.
The augmented matrix for this system of equations is as follows:
\[
M=
  \begin{bmatrix}
    -\alpha_2 & \alpha_1 & 0 & -\alpha_3 & \lambda  \\
    0 & -\alpha_3 & \alpha_2 & 0 & \mu  \\
    r \alpha_3 & 0 & -r \alpha_1 & \alpha_2 & \nu \\
    0 & 0 & 0 & \alpha_1 & \xi
  \end{bmatrix}.
\]
Performing row operations

(i) $R_3 \to \alpha_2 R_3$, (ii) $R_3 \to R_3 + r \alpha_3 R_1$ and (iii) $R_3 \to R_3 + r \alpha_1  R_2$,\\
the above matrix transforms to
\[
M_1=
  \begin{bmatrix}
    -\alpha_2 & \alpha_1  & 0 &  -\alpha_3 & \lambda  \\
    0 & -\alpha_3 & \alpha_2 & 0 & \mu \\
    0 & 0 & 0 &  \alpha_2^2 - r  \alpha_3^2 & \nu \alpha_2 +r \lambda
\alpha_3 + r \mu \alpha_1  \\
    0 & 0 & 0 & \alpha_1 & \xi
\end{bmatrix}.
\]
It is easy to see that the above system of equations admits a solution only if 
 $$\xi (\alpha_2^2 - r  \alpha_3^2 ) - \alpha_1 ( \nu \alpha_2 + r \lambda  \alpha_3 + r \mu \alpha_1  ) =0,$$
or equivalenty
 $$r \mu \alpha_1^2 + (\nu \alpha_2 + r \lambda \alpha_3) \alpha_1 - \xi(\alpha_2^2 - r \alpha_3^2) = 0.$$ 
 
Viewing the preceding equation as a quadratic equation in $\alpha_1$, notice that it has a solution in $\mathbb{F}_p$ if its  discriminant 
is zero or a quadratic residue mod $p$. The discriminant, after an easy computation, takes the form
 \begin{equation}\label{eqn21}
 (\nu^2 + 4r\mu \xi) \alpha_2^2 + 2r\lambda \nu \alpha_2 \alpha_3 +
r^2(\lambda^2 - 4 \mu \xi ) \alpha_3^2.
 \end{equation}
Notice that \eqref{eqn21} is of the form $f \alpha_2^2 + g \alpha_2 \alpha_3 + h \alpha_3^2$,
where $f,g,h \in \mathbb{F}_p$. As we know, there exist $\alpha_2, \alpha_3 \in \mathbb{F}_p$ such that \eqref{eqn21} is either zero or  a quadratic residue mod $p$. This completes the proof.  \hfill $\Box$

 \end{proof}
 
\begin{lemma} \label{cl2lemma2} 
Let $G$ be a  finite $p$-group of nilpotency class $2$ and order at least  $p^9$, $p \ge 3$,  such that $\gamma_2(G)$ is elementary abelian of order $p^4$ and $\Z(G) = \gamma_2(G)$. Then $\K(G)=\gamma_2(G)$. 
\end{lemma}
 \begin{proof} 
  If $b(G) = 4$, then we are done. So assume that $b(G) = 3$. Theorem \ref{psthm} now guarantees the existence of a normal subgroup $H$ of $G$ of order $p$ such that $|G/H : \Z(G/H)| = p^3$. Thus, using the fact that the exponent of $G/\Z(G)$ is $p$ (which follows from the given hypothesis), we can assume that 
 $$G=\langle a, b, c, x_1, \ldots , x_k\rangle,$$
 where $k \ge 2$ and $[x_i, G] = H$ for $1 \le i \le k$.  Let $S:=\gen{a, b, c}$ be the subgroup of $G$ generated by $a, b, c$. Notice that $|S|= p^6$ and $|\gamma_2(S)| = p^3$, and therefore $\K(S) = \gamma_2(S)$. If $S \le \C_G(x_i)$ for all $1 \le i \le k$, then, since  $\Z(G) = \gamma_2(G)$, $G$  can written as a central  product of $S$ and a $k$-generator group isoclinic to an extraspecial $p$-groups generated by $\{x_1, \ldots, x_k\}$. Now using Lemma \ref{prelemma4}, we have $\K(G) = \gamma_2(G)$. So assume that $[x_t,S] = H$ for some $t \in \{1, \ldots, k\}$. By reordering the set $\{x_1, \ldots, x_k\}$, if necessary, we can assume that $t=1$. For simplicity of notation, we set $d:= x_1$. Since $\C_G(d)$ is a maximal subgroup of $G$, we can modify the generators $a, b, c$ such that $\gamma_2(G) = \langle[a,b], [a,c], [b,c], [c,d]\rangle$, $[a, d] = [b, d] = 1$ and $H = \gen{[c,d]}$. We can also assume, by suitable modification of $x_i$, that $[c, x_i] = 1$ for all $2 \le i \le k$.
 
 Let $\alpha_i \in \mathbb{F}_p$ for $1 \le i\le 4$. If $\alpha_3 \neq 0$, then the following identity holds:
 $$[a,b]^{\alpha_1}  [b,c]^{\alpha_2} [a,c]^{\alpha_3} [c,d]^{\alpha_4}=[ c  b^{\alpha_1 {\alpha_3}^{-1}}, a^{-\alpha_3} b^{-\alpha_2} d^{\alpha_4}].$$ 
Thus any element of $\gamma_2(G)$ involving $[a,c]$ is a commutator. If $\alpha_3 = 0$ and $\alpha_2 \neq 0$, then for any $\alpha_1, \alpha_4 \in \mathbb{F}_p$ we have
$$ [a,b]^{\alpha_1}   [b,c]^{\alpha_2}    [c,d]^{\alpha_4}= [bd^{{-\alpha_4}{\alpha_2}^{-1}},   a^{-\alpha_1}  c^{\alpha_2}].$$
So, it only remains to show that elements of the form $[a,b]^{\alpha_1} [c,d]^{\alpha_4}$ are commutators, where both $\alpha_1$ and $\alpha_4$ are non-zero in $\mathbb{F}_p$. 

If  $[a,x_i] =[c,d]^{\beta}$ for some  $2\leq i \leq k$ and  $\beta \in \mathbb{F}_p^*$, then
$$[a,b]^{\alpha_1}[c,d]^{\alpha_4}=[a,b^{\alpha_1}{x_i}^{\beta^{-1}\alpha_4}].$$
Set $A := \gen{x_2, \ldots, x_k}$. Now suppose that $a \in \C_G(A)$, but $b \not\in \C_G(A)$. So there must exist an $x_i$, $2 \le i \le k$, such that  $[b,x_i] =[c,d]^{\beta}$, where $\beta \in \mathbb{F}_p^*$. Then
$$[a,b]^{\alpha_1}[c,d]^{\alpha_4}=[b,a^{-\alpha_1}{x_i}^{\beta^{-1}\alpha_4}].$$
Let $a, b \in \C_G(A)$.  If $[d,x_i] =[c,d]^{\beta}$ for some $2\leq i \leq k$ and $\beta \in \mathbb{F}_p^*$, then
$$[a,b]^{\alpha_1}[c,d]^{\alpha_4}=[ad,b^{\alpha_1}{x_i}^{\beta^{-1}\alpha_4}].$$
So finally assume that $a, b, c, d \in \C_G(A)$. Notice that, in this case, $k \ge 3$ and $\gamma_2(A) = H$, since $[A, G] = H$. Hence  $[x_i,x_j] =[c,d]^{\beta}$ for some $2\leq i, j \leq k$ and $\beta \in \mathbb{F}_p^*$, and therefore we have
$$[a,b]^{\alpha_1}   [c,d]^{\alpha_4}=[ bx_i, a^{-\alpha_1}   {x_j}^{ \alpha_4 \beta^{-1}}].$$
This shows that each element of $\gamma_2(G)$ is a commutator, and the proof is complete.  \hfill $\Box$

\end{proof}

\section{Groups of class $4$}

We start with groups of order $p^6$. Upto isoclinism, there are only 3 groups of order $p^6$ whose commutator subgroup is elementary abelian of order $p^4$, $p \ge 3$ (see \cite{rJ80}).

\begin{lemma}\label{p6lem}
Let $G$ be a  group of order $p^6$, $p \ge 3$ with $\gamma_2(G)$ is elementary abelian of order $p^4$. Then $\K(G) = \gamma_2(G)$ if and only if $|\Z(G)| = p$. Moreover, if $\K(G) \ne \gamma_2(G)$, then every element of $\gamma_2(G)$ can be written as a product of at most two elements from $\K(G)$.
\end{lemma}
\begin{proof}
It follows from \cite{rJ80} that upto isoclinism there are only three groups $G$ of order $p^6$, $p\ge 3$, such that $\gamma_2(G)$ is elementary abelian of order $p^4$. These fall under isoclinism families $\phi_{23}$, $\phi_{40}$ and $\phi_{41}$. All these groups are of nilpotency class $4$. For any group  $G$ belonging to  $\phi_{23}$, we have  $|\Z(G)| = p^2$ and it can be easily checked using GAP \cite{GAP} or Magma \cite{BCP} that $p \ge 5$. Let $G$ be a representative from $\phi_{23}$, which is presented as
\begin{eqnarray*}
 G &=& \langle\alpha, \alpha_1,  \alpha_2, \alpha_3, \alpha_4, \gamma \mid [\alpha_i,  \alpha]=\alpha_{i+1},  [\alpha_1,  \alpha_2]=\gamma,\\
 & & \;\  \alpha^p=\gamma, \alpha{_1}^p=\alpha_{i+1}^p=\gamma^p=1~ (i=1,2,3) \rangle.
\end{eqnarray*}
Notice that  $G$ is minimally generated by $\alpha$ and $\alpha_1$, exponent of $\gamma_2(G)$ is $p$ and $[\alpha_1, [\alpha, \alpha_1]] = \gamma \in Z(G)$. Hence it follows from \cite[Proposition 5.3]{KM05} that  $\K(G) \neq \gamma_2(G)$. More precisely, $\alpha_4 \gamma \notin \K(G)$.

For all groups $G$ from the isoclinism families $\phi_{40}$ and $\phi_{41}$, $|\Z(G)| = p$. If $p =3$, then an easy  GAP \cite{GAP} computation shows that $\K(G) = \gamma_2(G)$ for such groups $G$. So we can assume that $p \ge 5$.  Let $G$ be a representative from $\phi_{40}$, which is presented as
\begin{eqnarray*}
G &=&\langle \alpha_1, \alpha_2, \beta, \beta_1, \beta_2, \gamma \mid [ \alpha_1, \alpha_2]= \beta, [ \beta, \alpha_i]= \beta_i, [ \beta_1, \alpha_2]=[ \beta_2, \alpha_1]= \gamma,\\
& & \; \ \alpha^p=  {\alpha_i}^p= \beta^p={\beta_i}^p  =\gamma^p=1~ (i=1,2) \rangle.
\end{eqnarray*}
Notice that  $\gamma_2(G)=\langle \beta, \beta_1, \beta_2, \gamma \rangle$
and  $\gamma_4(G)=\langle \gamma \rangle$.
 Let $i,j,k \in \mathbb{F}_p $. If  $j \neq 0$, then
 $$\gamma^l = [  {\alpha_2}^{kj^{-1}}  \alpha_1 {\beta}^{\frac{1-i}{2}},  \beta_2^{-l} ]$$
 and  modulo $\gamma_4(G)$ we have
 $$  \beta^i \beta_1^j \beta_2^k  = [  {\alpha_2}^{kj^{-1}} {\alpha_1}
{\beta}^{\frac{1-i}{2}},  {\alpha_2}^i {\beta}^{-j} ].$$
If  $j = 0$, then
 $$\gamma^l= [ \alpha_2 {\beta}^{-\frac{1+i}{2}},  \beta_1^{-l} ]$$
 and  modulo $\gamma_4(G)$ we have
 $$  \beta^i \beta_2^k  = [ \alpha_2 {\beta}^{-\frac{1+i}{2}}, 
{\alpha_1}^{-i} \beta^{-k}].$$
Thus,  for $i, j, k \in \mathbb{F}_p$, it follows that
$$\gamma_2(G)/\gamma_4(G) = \big(  \bigcup \limits_{\substack{i , k \\ j \neq 0}}[ {\bar \alpha_2}^{\; kj^{-1}} {\bar \alpha_1} {\bar \beta}^{\; \frac{1-i}{2}} , \bar G]  \big) \bigcup  \big( \bigcup \limits_{ i} \;  [\bar \alpha_2 {\bar \beta}^{-\frac{1+i}{2}}, \bar G] \big),$$
where $\bar{\alpha}_1 = \alpha_1 \gamma_4(G)$, $\bar{\alpha}_2 = \alpha_2
\gamma_4(G)$ and $\bar{\beta} = \beta \gamma_4(G)$. Also
$$ \gamma_4(G) \subseteq \big(  \bigcap \limits_{\substack{i , k \\ j \neq 0}}[ {\alpha_2}^{kj^{-1}} {\alpha_1} {\beta}^{\frac{1-i}{2}} , G]  \big) \bigcap  \big( \bigcap \limits_{i} \; [\alpha_2 {\beta}^{-\frac{1+i}{2}}, G]  \big).$$
Hence $\K(G) = \gamma_2(G)$ by Lemma \ref{prelemma3}.

The group $G$, as presented below, is a representative from the isoclinism family $\phi_{41}$ for $p \ge 5$.
 \begin{eqnarray*}
 G &=& \langle \alpha_1, \alpha_2, \beta, \beta_1,
\beta_2, \gamma \mid [ \alpha_1, \alpha_2]= \beta, [ \beta, \alpha_i]=
\beta_i, [\alpha_1, \beta_1 ]={\alpha_1}^p= \gamma, \\
& &\;\ [ \alpha_2,\beta_2]=  {\gamma}^{-\nu},\alpha^p={\alpha_i}^p= \beta^p={\beta_i}^p 
=\gamma^p=1 \; (i=1,2) \rangle,
\end{eqnarray*}
where   $\nu$ denotes the smallest positive integer which is a non-quadratic residue (mod $p$). As in the preceding case, it is easy to check that $\gamma_2(G)/\gamma_4(G)$ can be written as the union of the sets $[\bar x, G/\gamma_4(G)]$, where $\bar x$ runs over the  elements of the set $S := \{{\bar \alpha_2}^{kj^{-1}} {\bar \alpha_1} {\bar \beta}^{\frac{1-i}{2}}, \bar \alpha_2 {\bar \beta}^{-\frac{1+i}{2}}  \mid i, j \;(\ne 0), k \in \mathbb{F}_p\}$ and $\gamma_4(G)$ is contained in the intersection of the sets $[x, G]$, where $\bar x \in S$. Hence $\K(G) = \gamma_2(G)$ by invoking Lemma \ref{prelemma3} again.  
The proof is now complete by taking $G/H$, where $H$ is any subgroup of $\Z(G)$ of order $p$, and using Theorem \ref{dmr77} and Lemma \ref{prelemma5}.  \hfill $\Box$
 
 \end{proof}

 Now we take up groups of order $p^7$.
 
  \begin{lemma} \label{p7lem3}
Let $L$ be a group of order $p^7$ and nilpotency class $4$ with $b(L)=3$, $\Z(L) \le \gamma_2(L)$  and $\gamma_2(L)$ elementary abelian of order $p^4$. Then $\K(L)=\gamma_2(L)$.
\end{lemma}
\begin{proof}
It follows from Lemma \ref{prelemma2} that $|\Z(L)| \le p^2$. By Theorem \ref{keyresult}, $L$ admits a subgroup $G$ of order $p^6$ and nilpotency class $4$ such that $\gamma_2(G) = \gamma_2(L)$. If  $|\Z(G)| = p$, then by Lemma \ref{p6lem} we have $\K(G) =  \gamma_2(G) = \gamma_2(L)$. Hence $\K(L) = \gamma_2(L)$.  So assume that $|\Z(G)| = p^2$. 
By Theorem \ref{psthm}, there exists a normal subgroup $H$ of $L$ such that  $|L/H:\Z(L/H)|=p^3$. We can take $L = \gen{a, b, c}$ such that $G = \gen{a, b}$.  As observed in the proof of Theorem \ref{keyresult} (last para), it follows that  $H = \gamma_4(L)$.  Thus $\bar{L} := L/H = \gen{\bar{a}, \bar{b}, \bar{c}}$ is of nilpotency class $3$ such that $\bar{c} \in \Z(\bar{L})$, where $\bar{x} = xH$ for any $x \in L$. Since $c \not\in \Z(L)$, we have $[c, L] = H$. Notice that $\gamma_2(\bar{L}) = \gen{[\bar a, \bar b], [\bar a, [\bar a, \bar b]],  [\bar b, [\bar a, \bar b]]}$. 
Let $\alpha_1,\alpha_2,\alpha_3 \in \mathbb{F}_p$. If $\alpha_3 \neq 0$, then we can write
 $$[\bar a, \bar b]^{\alpha_1}[\bar b,[\bar a, \bar b]]^{\alpha_2}[\bar a,[\bar a, \bar b]]^{\alpha_3} = [\bar a \ \bar b^{{\alpha_2}{\alpha_3}^{-1}} [\bar a, \bar b]^{\frac{1-\alpha_1}{2}}, \bar b^{\alpha_1}[\bar a, \bar b]^{\alpha_3}].$$
 If $\alpha_3 = 0$, then we can write
$$[\bar a, \bar b]^{\alpha_1} [\bar b,[\bar a, \bar b]]^{\alpha_2}=[ \bar b \ [\bar a, \bar b]^{- \frac{1+\alpha_1}{2}}, \bar a^{-\alpha_1}[\bar a, \bar b]^{\alpha_2} ].$$
Hence, for $i, j \in \mathbb{F}_p$ and $\epsilon = 0, 1$ such that  $i$ and $\epsilon$ are not simultaneously zero, we have
$$\gamma_2(\bar{L}) =   \bigcup \limits_{{\epsilon},{i}, {j}} \ [{\bar{a}}^{\epsilon} \ {\bar{b}}^i [\bar a, \bar b]^j, \bar L].$$
Notice that $G$ lies in the isoclinism family $\phi_{23}$ of \cite{rJ80}. Therefore we can take $H=\langle [b, \gamma_3(G)] \rangle$. Since $\C_G(c)$ is maximal,  by  suitably modifing $c$, we can assume that $[b,c]=1$ and $H=\langle [a,c] \rangle$. Then, for all $i, j \in \mathbb{F}_p$ and $\epsilon = 0, 1$,  it follows that $H \subseteq [a^{\epsilon} 
b^{i} [\bar a, \bar b]^j, L]$, where $i$ and $\epsilon$ are not simultaneously zero.
Hence $\K(L)=\gamma_2(L)$ by Lemma \ref{prelemma3}.  \hfill $\Box$

\end{proof}


 \section{Groups of class 3 and order $p^7$}\label{sec5}
 
In this section we take up groups $G$ of order $p^7$, $p$ odd, and nilpotency class $3$, and prove that  $\K(G)=\gamma_2(G)$ if $|\Z| \le p^2$ and $\K(G) \ne \gamma_2(G)$ otherwise. A group $G$ is said to be a freest group of nilpotency class $2$ on $n$ generators if $G = F_n/\gamma_3(F_n)$ and $|G| = p^{n(n+1)/2}$, where $F_n$ denotes the free group on $n$ generators. Throughout the remaining of the paper $\epsilon \in \{0, 1\}$. 

 \begin{lemma} \label{p7lem1}
Let $G$ be a group of order $p^7$ and nilpotency class $3$ with $b(G)=3$, $|\Z(G)| \le p^2$  and $\gamma_2(G)$ elementary abelian of order $p^4$. Then $\K(G)=\gamma_2(G)$.
\end{lemma}

\begin{proof} Notice that  $G$ is a stem group  in its isoclinism class.  In view of Lemma \ref{prelemma1},  it follows that $G$ is minimally generated by $3$ elements $a, \ b, \ c$ (say). For notational convenience, set $\C = \C_G(\gamma_2(G))$. To enhance the readability of the proof, we divide it in several steps.   

\;\;

\noindent{\bf Step 1.} {\it If $|\Z(G)| = p$, then $\C=\gamma_2(G)$.}
\begin{proof}
Since  $|Z(G)|=p$, we have $|\gamma_3(G)| = p$, and therefore no non trivial element from the subgroup $\gen{[a, b],  [a, c],  [b,c]}$ can lie in $\gamma_3(G)$.   If  $|\C|=p^6$, then, without loss of generality, we can assume that $b,c \in \C$. 
 By Hall-Witt Identity we have
   $$ [[a,b],c][[c,a],b][[b,c],a]=1,$$
 which implies that  $[b,c] \in \Z(G)$. As observed above, this is not possible. 
 
 If $|\C|=p^5$, then, without loss of generality, we can assume that $c \in \C$. If both $[a,[a,b]]$ and $[b,[a,b]]$ are trivial, then $[a,b] \in Z(G)$, which is not possible.  By symmetry, we can assume that $\gamma_3(G)=\langle [a,[a,b]] \rangle$.  Notice, by Hall-Witt Identity, that $[a,[b,c]] = [b,[a,c]]$. First we assume that $[a,[b,c]]$ is trivial. If $[b,[b,c]]$ or $[a,[a,c]]$ is trivial, then $[b,c]$ or $[a,c]$, respectively,  lies in $Z(G)$, which is again not possible. So assume that both $[b,[b,c]]$ and $[a, [a,c]] $ are non trivial. Then  $[b,[a,b]]=[b,[b,c]]^t$ for some $t \in \mathbb{F}_p$. Also $[a,[a,b]]=[a,[a,c]]^s$ for some $s \in \mathbb{F}_p^*$. Hence both $[b,[a,b][b,c]^{-t}]$ and $[a,[a,b][a,c]^{-s}]$ are trivial, which implies that  $[a,b][a,c]^{-s}[b,c]^{-t} \in Z(G)$, not possible. 
 
Finally assume that  $[a, [b,c]] $ is non trivial.  Then $[b,[a,c]]$ is also non trivial. If $[b, [b,c]]$ is non trivial, then $[a, [b,c]]=[b, [b,c]]^s$ for some $s \in  \mathbb{F}_p^*$,  and therefore $[ab^{-s}, [b,c]]=1$. By replacing $a$ by $ab^{-s}$, we get a modified generating set $\{a,b,c \}$ for $G$ such that $c \in \C$ and $[a,[b,c]]=1$.  Similarly, if $[a,[a,c]]$ is non trivial, then we can modify the generating set for $G$ such  that $[b, [a,c]]=1$. So, in both the cases, we land up in first case, which we have already handled. So let $[b, [b, c]] = [a, [a, c]] =1$. Then $[a, [a, b]] [a, [b, c]]^{-s} =1$, which implies that $[a, [a, b] [b, c]^{-s}] =1$ for some $s \in \mathbb{F}_p^*$. Similarly $[b, [a, b] [a, c]^{-t}] =1$ for some $t \in \mathbb{F}_p$. This implies that $[a, b] [b, c ]^{-s} [a, c]^{-t} \in \Z(G)$, which is not possible. Thus $\C$ can not have order $p^5$. Hence $\C =\gamma_2(G)$.
\end{proof}

\noindent{\bf Step 2.} {\it If $|\Z(G)| = p$, then $\K(G)=\gamma_2(G)$.}
\begin{proof}
By Step 1 we have $\C =\gamma_2(G)$.  Notice that  $\bar G := G/ \gamma_3(G)$ is the
$3$-generated freest group of nilpotency class $2$ and order $p^6$.
Let $\alpha_i  \in \mathbb{F}_p$, $1 \leq i \leq 3$. For $\alpha_1 \neq
0$, we have
$$[\bar a, \bar b]^{\alpha_1}  [\bar b, \bar c]^{\alpha_2}   [\bar a,  \bar c]^{\alpha_3}=
[\bar a \bar c^{{-\alpha_2}{\alpha_1}^{-1}}, \bar b^{\alpha_1} \bar c^{\alpha_3}],$$
where $\bar x = x \gamma_3(G)$ for  $x \in G$.
Also, for  $\alpha_1=0$,  we have
$$[\bar b, \bar c]^{\alpha_2}[\bar a, \bar c]^{\alpha_3}=[ \bar c, \bar a^{-\alpha_3} \bar b^{-\alpha_2}].$$
Hence for  $i \in \mathbb{F}_p$, we get
$$\gamma_2(\bar G) =  \bigcup \limits_{ \epsilon , i} \ [\bar a^{\epsilon} \bar c^i, \bar G].$$
Since $\bar G$ is the freest group,  by symmetry we can interchange $a,b,c$ in the preceding equation. 

To complete the proof of this step, it is sufficient to show that 
$$\gamma_3(G) \subseteq \bigcap \limits_{ \epsilon , i} \ [ x^{\epsilon} y^i,  G]$$
for some $x \ne y$  in $\{a, b, c\}$,  where   $i \in \mathbb{F}_p$ such that $\epsilon$ and $i$ are  not simultaneously zero. First assume that $[a,[a,b]] \ne 1$, and therefore generates $\gamma_3(G)$. 
If $[c,[a,b]] \neq 1$, then  $[c,[a,b]] = [a,[a,b]]^t$ for some $t \in \mathbb{F}_p^*$. Therefore, by modifying $c$ by $ca^{-t}$, we get a new generating set $\{a,b,c \}$ for $G$ such that $[c,[a,b]]=1$ and $[a,[a,b]] \ne 1$. So we can always assume that 
$[c,[a,b]]=1$. Since $c \notin \C$,  either $[c,[b,c]]$ or $[c,[a,c]]$ is non trivial. Hence for $\epsilon=0,1$ and $i
\in \mathbb{F}_p$, not simultaneously zero, it is easy to see that $\gamma_3(G) \subseteq    \bigcap \limits_{ \epsilon , i}
\ [a^{\epsilon} c^i, G]$.  

Now let us assume that $[a,[a,b]]=1$. Then, notice that, at least one of $[a,[b,c]]$ and $[a,[a,c]]$ is not trivial. 
If $[b,[a,b]]=1$, then $[c,[a,b]] \neq 1$.  Hence for $\epsilon=0,1$ and $i \in \mathbb{F}_p$ not simultaneously zero, we get $\gamma_3(G) \subseteq    \bigcap \limits_{ \epsilon , i} \ [c^{\epsilon} a^i, G]$. If $[b,[a,b]] \neq  1$, then $\gamma_3(G) \subseteq   \bigcap \limits_{ \epsilon , i} \ [b^{\epsilon} a^i, G]$, where both $\epsilon$ and $i$ are not simultaneously zero.
\end{proof}
 
 \noindent{\bf Step 3.} {\it If $|\Z(G)| = p^2$ and $|\gamma_3(G)| = p$, then $|\C|=p^5$.}
\begin{proof}
If $|\C| =p^6$, then, without loss of
generality, we can assume that $b,c \in \C$, and therefore by Hall-Witt Identity it follows that $[b,c] \in \Z(G)$. Hence no non-trivial element of the form $[a, b]^r [a, c]^s$ can lie in $\Z(G)$, where $r, s \in \mathbb{F}_p$.
If either  $[a,[a,b]]$ or $[a,[a,c]]$ is trivial,  then,  $\gamma_3(G)$ being of order $p$, it follows that either $[a,b]$ or $[a,c]$, respectively,  lie in $Z(G)$, which is not possible. If both $[a,[a,b]]$ and $[a,[a,c]]$ are non trivial, then  $[a,[a,b]]=[a,[a,c]]^t$ for some $t \in \mathbb{F}_p^*$. Hence $[a,b][a,c]^{-t} \in Z(G)$, which is again not possible. Hence $|\C| \ne p^6$.

If $|\C| = p^4$, then  $\C = \gamma_2(G)$. By a suitable modification in the generating set for $G$, we can assume that $[b,c] \in Z(G)$. Indeed, if one of $[a, b]$, $[a,c]$ and $[b, c]$ is in $\Z(G)$, then, after suitably renaming the generators, we can assume that $[b, c] \in \Z(G)$. If not, then, after renaming the generators, if necessary, we can assume that  $[b, c] = [a, b]^r[a, c]^s$ modulo $\Z(G)$ for some $r, s \in \mathbb{F}_p$. This implies that, modulo $\Z(G)$,  $[b a^{-s}, c a^r]=1$. Thus the new generating set $\{a,  b a^{-s}, c a^r\}$ has the required property. 
Once $[b, c] \in \Z(G)$, notice that  both $[a,[a, b]]$ and $[a, [a,c]]$ can not be trivial;  otherwise $a$ will lie in $\C$, which can not happen. By symmetry, we can assume that $[a,[a,b]]$ is non trivial and therefore generates $\gamma_3(G)$. By Hall-Witt Identity we have $[c,[a,b]]=[b,[a,c]]$. We first assume that $[c,[a,b]]$ is trivial.  Then none of the two elements $[c,[a,c]]$ and  $[b,[a,b]]$ can be trivial. Indeed, if $[x, [a,x]] = 1$, then $x \in \C$ for $x = b, c$, which is not the case. Then $[a,[a,c]]=[c,[a,c]]^r$ for some $r \in \mathbb{F}_p$, which implies that $[a c^{-r},[a,c]]=1$.  Replacing $a$ by $ac^{-r}$, we get a generating set $\{a' = a c^{-r}, b, c \}$ such that $\gamma_2(G)=\langle [a',b],[a',c],[b,c],[a',[a',b]]\rangle$. A straightforward computation shows that  $[b,c] \in Z(G)$   and $[b,[a', c]] = [c,[a',b]] =1 = [a',[a',c]]$; but $[c,[a',c]]$ and  $[b,[a',b]]$ are non trivial. Now $[a',[a',b]]=[b,[a',b]]^t$ for some $t \in \mathbb{F}_p^*$,  which implies that $a' b^{-t}$ lies in $\C$, which is not possible.

Now we claim that when at least one the two elements $[b, [a,b]]$ and $[c, [a,c]]$ is non trivial, then we can modify the generating set for $G$ such that $[c,[a,b]] =1$, but $[b,[a,b]]$ and $[c, [a,c]]$ are both non trivial, and hence we fall in the preceding case. If $[b,[a,b]]$ is non trivial, then $[c,[a,b]] = [b,[a,b]]^t$ for some $t \in \mathbb{F}_p$. Replacing $c$ by $b^{-t}c$, we get the required generating set. If $[c, [a,c]]$ is non trivial,  then $[b,[a,c]] =[c,[a,c]]^r$ for some $r \in \mathbb{F}_p$. Now  replacing  $b$ by $c^{-r}b$, we again get the required generating set. 

Finally assume that $[b, [a, c]] = [c, [a, b]]$ is non trivial, and  $[b,[a,b]]$ and $[c,[a,c]]$ are both trivial. Then $[a,[a,c]]=[b,[a,c]]^r$ for some $r \in \mathbb{F}_p$, which gives $ [ab^{-r}, [a,c]]=1$. Replacing $a$ by $ab^{-r}$, we get a generating set 
$\{a'=ab^{-r} ,b,c \}$ for $G$ such that $\gamma_2(G)=\langle [a',b],[a',c],[b,c],[a',[a',b]]\rangle$, with $[b,c] \in Z(G)$  and $[a',[a',c]]=1$. Notice that  $[a', [a',b]] = [a,[a,b]]$ and $[c, [a',b]] = [c,[a,b]]$, both of which are non trivial. Hence $[a',[a',b]] =
[c,[a',b]]^t$ for some $t \in \mathbb{F}_p^*$, which gives $[a'c^{-t}, [a',b]]=1$. Since $[c, [a',c]] = [c, [a,c]] = 1$, it follows that  
$[a'c^{-t}, [a', c]]=1$. Hence $a'c^{-t} \in \C$, which is not possible. We have handled all the cases, and the proof of this step is complete.
\end{proof}

 \noindent{\bf Step 4.} {\it If $|\Z(G)| = p^2$ and $|\gamma_3(G)| = p$, then $\K(G) = \gamma_2(G)$.}
\begin{proof}
By Step 3, we know that  $|\C| = p^5$. Without loss of generality we can
assume that  $c \in C$. We consider three different cases, namely $[b,c]
\in \Z(G)$,  $[a,c] \in \Z(G)$ and otherwise.
Since $\bar G :=  G/\gamma_3(G)$ is the freest group, then, as explained in  Step 2, for proving $\K(G) = \gamma_2(G)$, it is sufficient to show that
$$\gamma_3(G) \subseteq \bigcap \limits_{ \epsilon , i} \ [ x^{\epsilon} y^i,  G]$$
for some $x \ne y$  in $\{a, b, c\}$,  where  $\epsilon=0,1$ and $i \in
\mathbb{F}_p$ such that $\epsilon$ and $i$ are  not simultaneously zero.

{\it Case (i).} Let $[b,c] \in \Z(G)$. Notice that all $[a, [b,c]]$, $[c,
[a, b]]$ and $[b, [a, c]]$ are trivial. Then none of the two elements $[a,
[a, c]]$ and $[b, [a, b]]$ can be trivial. Thus for any $ t , i  \in
\mathbb{F}_p$ we have
 $[a,[a,c]]^t=[ab^i , [a,c]^t]$ and  $[b,[a,b]]^t=[b , [a,b]^t],$ which implies that
$$ \gamma_3(G) =  \bigcap \limits_{ \epsilon, i} \ [a^{\epsilon} b^{i}, \gamma_2(G)] \subseteq \bigcap \limits_{ \epsilon, i} \ [a^{\epsilon} b^{i}, G],$$
where $\epsilon$ and $i$ are  not simultaneously zero.
Hence  $\K(G) = \gamma_2(G)$ by Lemma \ref{prelemma3}.

{\it Case (ii).}   Next assume that $[a,c] \in \Z(G)$. As in the preceding
case, both the elements $[a, [a, b]]$ and $[b, [b, c]]$ are non trivial
and $[a, [b, c]]=1$. For any  $ t, i \in \mathbb{F}_p$ we have
$[b, [b,c]]^t=[ba^i , [b,c]^t]$
 and
 $[a,[a,b]]^t=[a , [a,b]^t].$ Thus for $\epsilon = 0, 1$ and $i \in \mathbb{F}_p$, not simultaneously zero, we get
$$ \gamma_3(G) = \bigcap \limits_{ \epsilon, i} \ [b^{\epsilon}
a^{i}, \gamma_2(G)] \subseteq \bigcap \limits_{ \epsilon, i} \ [b^{\epsilon} a^{i}, G].$$

{\it Case (iii).} Let neither $[b, c]$ nor $[a, c]$ lie in $\Z(G)$. We can
assume that none of the two elements $[b, c]$ and $[a, c]$ can be a
power of the other. For, if  $[b, c] = [a, c]^t$, for some $t \in
\mathbb{F}_p^*$, then we can take a generating set $\{a, ba^{-t}, c\}$
such that $[ba^{-t}, c]  \in \Z(G)$,  and we fall in Case (i). We now
modify the generating set for $G$ to $\{a', b', c\}$ such that  $[a', b']
\in \Z(G)$ and $c \in \C$. If $[a, b] \in \Z(G)$, then obviously we take
$\{a' = a, b' = b, c\}$. If not, then, modulo $\Z(G)$,  $[a, b] = [a,
c]^r[b, c]^s$ for some $r, s \in \mathbb{F}_p$. This implies that, modulo
$\Z(G)$,  $[ ac^{s}, b c^{-r}]=1$. Thus if we take $\{a'= ac^s,  b' =
bc^{-r}, c\}$ as a generating set for $G$, then $[a', b'] \in \Z(G)$ and
$c$ still lies in $\C$.

First assume that $[a',[a',c]]=[b',[b',c]]=1$. Since $a', b' \notin C$, we
have $[a', [b',c]] = [b', [a',c]] \neq 1$. For any  $ t, i  \in
\mathbb{F}_p$ we have
 $[a', [b' ,c]]^t=[a'b'^i , [b', c]^t]$
 and
 $[b',[a',c]]^t=[b' , [a',c]^t].$
Thus for $\epsilon = 0, 1$ and $i \in \mathbb{F}_p$, not simultaneously zero, we get
$$ \gamma_3(G) = \bigcap \limits_{ \epsilon, i} \ [a^{\epsilon} b^{i}, \gamma_2(G)] \subseteq  \bigcap \limits_{ \epsilon, i} \ [a^{\epsilon}
b^{i}, G].$$

Now assume that  $[b',[b',c]]$ or $[a',[a',c]]$ is non trivial. We only
present the proof when $[b',[b', c]] \neq 1$. The other case goes on
similar lines. Without loss of generality, we can  assume
$[a',[b',c]]=[b',[a',c]]= 1$. Indeed, if $[b',[a',c]]= 1$, then nothing to
be done. Otherwise
$[a',[b',c]] = [b',[b',c]]^s$, which gives
$[a' {b'}^{-s},[b',c]]=1$. Replacing $a'$ by  $a'{b'}^{-s}$, we get the
required generating set $\{\tilde{a} = a'{b'}^{-s}, b', c\}$ for which  
$[b', [\tilde{a},c]]=[\tilde{a},[b',c]]=1$, $[\tilde{a},b'] \in Z(G)$, $c
\in C$ and $[\tilde{a},[\tilde{a},c]] \neq1$.
For any $ t, i \in \mathbb{F}_p$ we have
$[\tilde{a}, [ \tilde{a} ,c]]^t =[\tilde{a} b'^i , [\tilde{a},c]^t]$
 and
 $[b',[b', c]]^t=[b' , [b',c ]^t].$ Thus for $\epsilon = 0, 1$ and $i \in \mathbb{F}_p$, not simultaneously zero, we get
 $$ \gamma_3(G) = \bigcap \limits_{ \epsilon, i} \ [a^{\epsilon} b^{i}, \gamma_2(G)] \subseteq  \bigcap \limits_{ \epsilon, i} \ [a^{\epsilon}
b^{i}, G].$$
Hence $\K(G) = \gamma_2(G)$, which completes the proof of Step 4.
\end{proof}

It only remains to handle the situation when $|\Z(G)| = |\gamma_3(G)| = p^2$.  We claim that $b(G) = 4$ in this case. Contrarily assume that $b(G) = 3$. Then  by Theorem \ref{psthm} there exists a normal subgroup $H$ of $G$ such that $|G/H : \Z(G/H)| = p^3$.  Since $\Z(G) = \gamma_3(G)$ is of order $p^2$, the nilpotency class of $G/H$ is $3$. So we can assume that $cH \in \Z(G/H)$.  Notice that $|\gamma_2(G/H) /\gamma_3(G/H)|$ must be  $p^2$, which is not possible as $G/H$ admits only two non-central generators. Hence the claim follows, and the proof of the lemma is complete \hfill $\Box$
  
 \end{proof}

We now prove

\begin{lemma}\label{p7lem2}
Let $G$ be a group of order $p^7$ and nilpotency class $3$ with $b(G)=3$,
$|Z(G)| = p^3$  and $\gamma_2(G)$ elementary abelian of order $p^4$, where $p$ is an odd prime.
Then $\K(G) \ne \gamma_2(G)$. Moreover,  every element of $\gamma_2(G)$ can be written as a product of at most two elements from $\K(G)$.
\end{lemma}
 \begin{proof}
We start by noticing that $|\gamma_3(G)| \le p^2$. Let $G = \gen{a, b, c}$. Then $\gamma_2(G)= \langle [a,b], [b,c], [a,c], $ $\gamma_3(G) \rangle$. Without loss of generality we can assume that $Z(G)= \langle [b,c], [a,c], \gamma_3(G) \rangle$. Indeed, modulo $\Z(G)$, we can assume that $c$ commutes with $a$ and $b$.  This, by Hall-Witt Identity, implies that  $c \in \C_G(\gamma_2(G))$. Now we consider two cases, namely, $|\gamma_3(G)|$ is $p$ or $p^2$.

First assume  $|\gamma_3(G)| = p$. By suitably modifying the generating set $\{a, b, c\}$, if necessary, we can assume that $[a,[a,b]] \neq 1$ and $[b,[a,b]] = 1$. Hence  $\gamma_3(G) = \gen{[a,[a,b]]}$. Now we claim that $[b,c][a,[a,b]]$ is not a commutator.
Contrarily  assume that 
$$[b,c][a,[a,b]] = [a^{\alpha_1}{b}^{\alpha_2}c^{\alpha_3}[a,b]^{\alpha_4}, a^{\beta_1}{b}^{\beta_2}c^{\beta_3}[a,b]^{\beta_4}],$$
where $\alpha_i, \beta_i \in \mathbb{F}_p$. After solving  and comparing the powers, we get
\begin{eqnarray}
\beta_2 \alpha_1 - \alpha_2 \beta_1 & = & 0, \label{eqnn1} \\
\beta_3 \alpha_2 - \alpha_3 \beta_2 &=& 1, \label{eqnn2}\\
\beta_3 \alpha_1 - \alpha_3 \beta_1 &=& 0, \label{eqnn3}\\
\alpha_2{ \binom {\beta_1}2} - \beta_2{\binom {\alpha_1}2} + \beta_4 \alpha_1 - \alpha_4 \beta_1 &=&1.\label{eqnn4}
\end{eqnarray}
If  $\alpha_2 = 0$, then from \eqref{eqnn2}, we get $\alpha_3, \beta_2 $ are non zero. So \eqref{eqnn1} gives  $\alpha_1=0$, which in turn, using \eqref{eqnn3} gives  $\beta_1=0$. But these values contradict \eqref{eqnn4}.
Now let $\alpha_2 \neq 0$.  Then from \eqref{eqnn1},  $\beta_1= \beta_2 \alpha_1 {\alpha_2}^{-1}$. Substituting $\beta_1$ in \eqref{eqnn4}, we  get
$$\alpha_1\big(\beta_4 +\alpha_2^{-1} \beta_2 ((\alpha_1 \beta_2 - \alpha_2 \beta_2)/2-\alpha_4)\big)=1;$$ 
hence $\alpha_1
\neq 0$.   Substituting $\beta_1$ in \eqref{eqnn3}, we get
$\alpha_1(\beta_3-\alpha_3\beta_2{\alpha_2}^{-1})=0$. Hence
$\beta_3=\alpha_3\beta_2{\alpha_2}^{-1} $, which contradicts \eqref{eqnn2}. Our claim is now settled.

Now we assume $|\gamma_3(G)|=p^2$.  Thus $\gamma_3(G)= \langle [a,[a,b]],[b,[a,b]]\rangle$.
Notice that if neither $[a, c]$ nor $[b, c]$ lies in $\gamma_3(G)$, then one of them will be a power  of the other modulo $\gamma_3(G)$. Then we can modify the generating set $\{a, b, c\}$ such that $[a, c] \in \gamma_3(G)$, without disturbing other setup.
Let $[a,c]=[a,[a,b]]^{t_1}[b,[a,b]]^{t_2}$ for some $t_1,t_2 \in \mathbb{F}_p$. Then $[a,c[a,b]^{-t_1}]=[b,[a,b]]^{t_2}$. By replacing
$c$ with $c[a,b]^{-t_1}$, we get a modified generating set for $G$, which we still call $\{a, b, c\}$, such that
$[a,c]=[b,[a,b]]^{t_2}$ and $c \in \C_G(\gamma_2(G))$. If $t_2 =0$, then $[a,c] = 1$, otherwise we can replace $c$ by $c^{t_2^{-1}}$, and assume that $t_2 = 1$; hence $[a,c]=[b,[a,b]]$.

Let us assume that for given $\lambda, \nu \in \mathbb{F}_p^*$, there exist $\alpha_i, \beta_i \in \mathbb{F}_p$ such that
$$[b,c]^{\lambda}[a,[a,b]]^{\nu} = [a^{\alpha_1}b^{\alpha_2}{c}^{\alpha_3}[a,b]^{\alpha_4}, a^{\beta_1}b^{\beta_2}{c}^{\beta_3}[a,b]^{\beta_4}].$$
 After solving  and comparing the powers on both sides, we get
\begin{eqnarray} \label{numeric6}
\beta_2 \alpha_1 - \alpha_2 \beta_1& =& 0 \label{numeric6}\\
\beta_3 \alpha_2 - \alpha_3 \beta_2 &=& \lambda  \label{numeric7}\\
\beta_1\beta_2\alpha_2-\alpha_1\alpha_2\beta_2-\alpha_1{ \binom
{\beta_2}2} + \beta_1{\binom {\alpha_2}2} + (\beta_3 \alpha_1 - \alpha_3
\beta_1)+\beta_4 \alpha_2 - \alpha_4 \beta_2&=&0 \label{numeric8}\\
\alpha_2{ \binom {\beta_1}2} - \beta_2{\binom {\alpha_1}2}  +\beta_4
\alpha_1 - \alpha_4 \beta_1&=&\nu.
\label{numeric9}
\end{eqnarray}
Using  \eqref{numeric6}  in (\ref{numeric8}) and \eqref{numeric9}, we, respectively,  get
\begin{eqnarray}
( \beta_1\beta_2\alpha_2-\alpha_1\alpha_2\beta_2) / 2 + (\beta_3 \alpha_1
- \alpha_3 \beta_1)+\beta_4 \alpha_2 - \alpha_4 \beta_2&=& 0
\label{numeric10}\\
( \beta_1\beta_2\alpha_1-\alpha_1\alpha_2\beta_1) / 2 +\beta_4 \alpha_1 -
\alpha_4 \beta_1&=&\nu.
\label{numeric11}
\end{eqnarray}

 We proceed in two different cases, namely, $[a, c] = 1$ and $[a,c]=[b,[a,b]]$.

{\it Case(1).} Let $[a, c] = 1$. Then \eqref{numeric10} reduces to
\begin{equation}\label{numeric12}
( \beta_1\beta_2\alpha_2-\alpha_1\alpha_2\beta_2) / 2 + \beta_4 \alpha_2 - \alpha_4 \beta_2 =  0.
\end{equation}
 If $\alpha_2=0$, then by \eqref{numeric7}, we get $\beta_2 \ne 0$, which using,  \eqref{numeric6}, gives $\alpha_1=0$. Hence  by \eqref{numeric12},  $\alpha_4=0$, which contradicts \eqref{numeric11}. If $\alpha_2 \neq 0$ and  $\alpha_1=0$, then by \eqref{numeric6}
$\beta_1=0$, which contradicts \eqref{numeric11}. Hence,  $\alpha_2 \neq 0$ implies $\alpha_1 \ne 0$. By symmetry we can take both $\beta_1$ and $\beta_2$  non zero.  Hence we can assume that  $\alpha_1,\alpha_2, \beta_1, \beta_2$ are all nonzero. 

Computing the value of $\alpha_4$ from \eqref{numeric12} and substituting in \eqref{numeric11}, we get
\begin{equation}
( \beta_1\beta_2\alpha_1-\alpha_1\alpha_2\beta_1) / 2+\beta_4 \alpha_1 - (( \beta_1\alpha_2-\alpha_1\alpha_2) / 2+\beta_4 \alpha_2{\beta_2}^{-1}) \beta_1=\nu. \nonumber
\end{equation}
Using \eqref{numeric6}, it is easy to see that the left hand side of the
preceding equation is zero, which contradicts the choice of $\nu$. Hence for any $\lambda, \nu \in \mathbb{F}_p^*$, $[b,c]^{\lambda}[a,[a,b]]^{\nu}$ is not a commutator.

{\it Case(2).}  Let $[a,c]=[b,[a,b]]$.  If $\alpha_2=0$, then by \eqref{numeric6}, $\alpha_1=0$, and the  equations \eqref{numeric7}, \eqref{numeric10} and \eqref{numeric11}, respectively, reduces to
\begin{eqnarray*}
 \alpha_3\beta_2 &=& -\lambda\\
\alpha_3\beta_1+\alpha_4\beta_2 &=& 0\\
\alpha_4\beta_1 &=& -\nu.
\end{eqnarray*}
Solving this we get $ ({\beta_1}^{-1}\beta_2)^2 = - \lambda \nu^{-1}$, which is not possible for a choice of $\lambda$ and $\mu$ such that $- \lambda \nu^{-1}$ is a non-quadratic residue mod $p$. Hence  for any such $\lambda,  \nu \in \mathbb{F}_p^*$, 
$[b,c]^{\lambda} [a,[a,b]]^{\nu}$ is not a commutator. If $\alpha_2 \neq 0$ and  $\alpha_1=0$, then by \eqref{numeric6}
$\beta_1=0$, which contradicts \eqref{numeric11}. Hence,  $\alpha_2 \neq 0$ implies $\alpha_1 \ne 0$. By symmetry we can take both $\beta_1$ and $\beta_2$  non zero.  Hence we can assume that  $\alpha_1,\alpha_2, \beta_1, \beta_2$ are all nonzero. 

Using  \eqref{numeric6} in \eqref{numeric10} and  \eqref{numeric11},  we,  respectively,  get   
$$\alpha_4=( \beta_1\alpha_2-\alpha_1\alpha_2) / 2 + {\beta_2}^{-1}(\beta_3 \alpha_1 - \alpha_3 \beta_1)+\beta_4 \alpha_2{\beta_2}^{-1}$$
and 
$$\alpha_4=(\beta_2\alpha_1-\alpha_1\alpha_2) / 2+ \beta_4 \alpha_1 {\beta_1}^{-1} -
\nu{\beta_1}^{-1}.$$
Equating these equations and using \eqref{numeric6}, we get $(\beta_3 \alpha_1 - \alpha_3 \beta_1)= - \nu{\beta_1}^{-1}\beta_2$. Multiply both sides by
${\beta_1}^{-1}\beta_2$ and  using \eqref{numeric6} and \eqref{numeric7}, we further get 
$({\beta_1}^{-1}\beta_2)^2= - \lambda\nu^{-1}$. Hence, as above, for $\lambda$ and $\nu$ such that $- \lambda \nu^{-1}$ is a non-quadratic residue mod $p$,
$[b,c]^{\lambda} [a,[a,b]]^{\nu}$ is not a commutator. The proof is now complete by taking $G/H$, where $H$ is any subgroup of $\Z(G)$ of order $p$, and using Theorem \ref{dmr77} and Lemma \ref{prelemma5}. \hfill $\Box$

\end{proof}

 \section{Proof of Theorem A}
 
 In this section we present a proof of Theorem A. The following result is the last thread to accomplish our endeavour.
 \begin{lemma}\label{main1}
 Let $L$ be a finite $p$-group of order at least $p^8$ and nilpotency class $3$ with $b(L)=3$, $|\Z(L)| = p^3$, $\Z(L) \le \gamma_2(L)$  and $\gamma_2(L)$ elementary abelian of order $p^4$. Then $\K(L)=\gamma_2(L)$.
\end{lemma}
\begin{proof}
Notice that $|\gamma_3(L)|  \le p^2$. We first assume that $|\gamma_3(L)| = p$. Then by Theorem \ref{psthm} there exists a normal subgroup $H$ of $L$ such that $|L/H : \Z(L/H)| = p^3$. If $H \neq \gamma_3(L)$, then it follows from Theorem \ref{keyresult} that $L$ admits a $2$-generator subgroup $G$ such that $\gamma_2(G) = \gamma_2(L)$, which is not possible as 
$|\gamma_2(L)/\gamma_3(L)| = p^3$. Hence $H = \gamma_3(L)$, and therefore, again by Theorem \ref{keyresult}, $L$ admits a $3$-generator subgroup $G$ of order $p^7$ such that $\gamma_2(G) = \gamma_2(L)$.

 If $|L| = p^8$, then $L = \gen{a, b, c, d}$ such that $G = \gen{a, b, c}$ and $[d, G] = H$.  If $|L| \ge p^9$, then, for some integer $k \ge 2$, $L = \gen{a, b, c, x_1, \ldots, x_k}$ such that $G = \gen{a, b, c}$ with $\gamma_2(G) = \gamma_2(L)$ and $[x_i, L] = H$ for $1 \le i \le k$.  First assume that $[x_i, G] = H$ for some $1 \le i \le k$. Then the subgroup $M:=\gen{a, b, c, d}$, where $d = x_i$, is of order $p^8$ such that $[d, G] = H$ and $\gamma_2(M) = \gamma_2(L)$. Hence it is sufficient to work with $M$. So this case reduces to the preceding situation when $|L|= p^8$.

It now follows from the proof of Lemma \ref{p7lem2} that we can modify the generating set for $G$  such that $\gamma_2(L) = \gamma_2(G)=\langle [a,b],[a,c],[b,c],[a,[a,b]]\rangle$ and $[c, [a, b]] =1$. Notice that $L/H$ is of nilpotency class  $2$, and therefore $\gamma_2(L/H)=\langle [\bar a, \bar b],[\bar a, \bar c],[\bar b, \bar c]\rangle$, where $\bar x = xH$ for all $x \in L$. Also notice that $\bar d \in \Z(L/H)$.

If $[a,d]=[a,[a,b]]^t$ for some $t \in \mathbb{F}_p$, then $[a,d[a,b]^{-t}]=1$. Replacing $d$ by $d[a,b]^{-t}$, we can assume that $[a,d]=1$. If $[c, d] \ne 1$, then $[c, d] = [a, [a, b]]^r$ for some $r \in \mathbb{F}_p^*$.  Let $\alpha_i \in \mathbb{F}_p$, $1 \le i \le 4$. If  $\alpha_1 \ne 0$, then modulo $H$ we can write 
$$[a,b]^{\alpha_1}[b,c]^{\alpha_2}[a,c]^{\alpha_3}=[ac^{{-\alpha_2}{\alpha_1}^{-1}}, b^{\alpha_1}c^{\alpha_3}].$$
And
$$[a,[a,b]]^{\alpha_4}=[ac^{{-\alpha_2}{\alpha_1}^{-1}}, [a,b]^{\alpha_4}].$$
If $\alpha_1=0$, then modulo $H$ we can write
$$[b,c]^{\alpha_2}[a,c]^{\alpha_3}=[ c, a^{-\alpha_3}b^{-\alpha_2}].$$
And
$$[a,[a,b]]^{\alpha_4}=[c,d^{{\alpha_4}r^{-1}}].$$
Thus for $i \in \mathbb{F}_p$ and $\epsilon = 0, 1$, we get
$$\gamma_2(L)/H = \bigcup \limits_{\epsilon, i} \ [{\bar a}^{\epsilon} \bar c^{i}, \bar L]$$
and
$$H \subseteq \bigcap \limits_{\epsilon, i}  \ [a^{\epsilon}  c^{i}, L],$$
where $\bar x = xH$ for $x \in L$, and $i$ and $\epsilon$ are not simultaneously zero.
Hence, by Lemma \ref{prelemma3}, $\K(L) = \gamma_2(L)$. Finally, if $[c, d] =1$, then $[b , d] \ne 1$. The assertion now  follows on the same lines as above.

Now we consider the remaining case for the group $L$ with $|L| \ge p^9$, i.e., $[x_i, G] = 1$ for all $1 \le i \le k$. Then $L$ is a central product of $G$ and $K := \gen{x_1, \ldots, x_k}$ amalgamating  $H$. Since $\Z(L) \le \gamma_2(L)$, $K$ can not be abelian. Thus $H = \gen{[x_i, x_j]}$ for some $1 \le i < j \le k$. Then the subgroup $N:=\gen{a, b, c, d, e}$, where $d = x_i$ and $e = x_j$, is of order $p^{9}$ such that $\gamma_2(N) = \gamma_2(L)$. Now on the lines of the preceding case, for $i \in \mathbb{F}_p$, it is not difficult to see that 
$$\gamma_2(N)/H = \bigcup \limits_{i} \ [ \bar a  \bar c^{i} \bar d, \bar N] \ \bigcup \ [\bar c \bar d, \bar N] $$
and
$$H \subseteq \bigcap \limits_{ i} \ [ a  c^{i} d, N]  \ \bigcap \ [c d, N],$$
where  $\bar x = x H$ for $x \in N$. Hence, again by Lemma \ref{prelemma3}, $\K(N) = \gamma_2(N)$.

Now assume that  $|\gamma_3(L)| = p^2$. As in the above case, we can show that $H \not\le \gamma_3(L)$. Then, by Theorem \ref{keyresult}, either (i) $L$  admits a $2$-generator subgroup $G$ of order $p^5$ and nilpotency class $3$ such that $L$ is a central product of $G$ and subgroup $K$ with $|\gamma_2(K)| = p$ or (ii) $L$ admits a $3$-generator subgroup $G$ of  order $p^7$ and nilpotency class $3$ such that $\gamma_2(G) = \gamma_2(L)$ and  $L$ is an amalgamated semidirect product of $G$ and a subgroup $K$ of nilpotency class at most $2$. In  case  (i), it follows from Lemma \ref{prelemma4} that $\K(L) = \gamma_2(L)$. So assume (ii).

If $|L| = p^8$, then $L = \gen{a, b, c, d}$ such that $G = \gen{a, b, c}$ and $\gen{[d, G]} = H$. If $|L| \ge p^9$, then, for some integer $k \ge 2$, $L = \gen{a, b, c, x_1, \ldots, x_k}$ such that $G = \gen{a, b, c}$ and $\gen{[x_i, L]} = H$ for $1 \le i \le k$. 
First assume that $[x_i, G] = H$ for some $1 \le i \le k$. Then the subgroup $M:=\gen{a, b, c, d}$, where $d = x_i$, is of order $p^8$ such that $[d, G] = H$ and $\gamma_2(M) = \gamma_2(L)$. As observed above, it is sufficient to work with the situation when $|L|= p^8$. 

As explained in the proof of Lemma \ref{p7lem2}, we can assume that 
$H= \gen{[b, c]}$, $$\gamma_3(G) = \gen{[a, [a, b]], [b, [a,b]]},$$
 $[a, c] \in \gamma_3(G)$ and $c \in \C_G(\gamma_2(G))$.  If $[b,d]=[b,c]^t$ for some $t \in \mathbb{F}_p$, then $[b,dc^{-t}]=1$. Set $d' = dc^{-t}$. 
If $[c,d']=[b,c]^s$ for some $s \in \mathbb{F}_p$, then $[c, d'b^s] =1$. Replacing $d$ by $d'b^s$, we can assume that $[b, d] = [c, d] = 1$.  Let $[a,d] =[b,c]^r$ for some $r \in \mathbb{F}_p^*$. Let $\alpha_1,\alpha_2,\alpha_3,\alpha_4 \in \mathbb{F}_p$. If $\alpha_3 \neq 0$, then, modulo $H$ we can write
 $$[a,b]^{\alpha_1}[b,[a,b]]^{\alpha_2}[a,[a,b]]^{\alpha_3} = [a b^{{\alpha_2}{\alpha_3}^{-1}}  [a, b]^{\frac{1-\alpha_1}{2}}, b^{\alpha_1}[a,b]^{\alpha_3}].$$
 And 
$$[b,c]^{\alpha_4}=[a b^{{\alpha_2}{\alpha_3}^{-1}}  [a, b]^{\frac{1-\alpha_1}{2}},d^{{\alpha_4}r^{-1}}].$$
If $\alpha_3 = 0$, then modulo $H$ we can write
$$[a,b]^{\alpha_1}[b,[a,b]]^{\alpha_2}=[ b [a, b]^{- \frac{\alpha_1 +1}{2}},a^{-\alpha_1}[a,b]^{\alpha_2}
].$$
And
$$[b,c]^{\alpha_4}=[b [a, b]^{- \frac{\alpha_1 +1}{2}}, c^{\alpha_4}].$$
Thus, for  $i, j \in \mathbb{F}_p$ and $\epsilon = 0, 1$ such that  $i$ and $\epsilon$ are not simultaneously zero,  we get
$$\gamma_2(L)/H = \bigcup \limits_{ \epsilon, i, j} \ [ \bar a^{\epsilon} {\bar b}^{i}  [a, b]^j, \bar L]$$
and
$$H \subseteq  \bigcap \limits_{ \epsilon, i, j} \ [ a^{\epsilon} b^{i}   [a, b]^j, L] ,$$
 where  $\bar x = x H$ for $x \in L$.
Hence $\K(L) = \gamma_2(L)$.

Finally assume that $[x_i, G] = 1$ for all $1 \le i \le k$. In this case, $L$ is a central product of $G$ and $K := \gen{x_1, \ldots, x_k}$ amalgamating some subgroup containing $H$. Since $\Z(L) \le \gamma_2(L)$, $K$ can not be abelian. Thus $H = \gen{[x_i, x_j]}$ for some $1 \le i < j \le k$. Then the subgroup $N:=\gen{a, b, c, d, e}$, where $d = x_i$ and $e = x_j$, is of order $p^{9}$ such that $\gamma_2(N) = \gamma_2(L)$. For $i \in \mathbb{F}_p$, it is again easy to see that 
$$\gamma_2(N)/H = \bigcup \limits_{i} \ [\bar a  \bar b^{i}  \bar d \; [\bar a, \bar b]^{\frac{1-\alpha_1}{2}}, \bar N] \ \bigcup \ [\bar b \bar d \; [\bar a, \bar b]^{ - \frac{\alpha_1 +1}{2}}, \bar N] $$
and
$$H \subseteq \bigcap \limits_{ i} \ [a b^{i} d \; [a, b]^{\frac{1-\alpha_1}{2}}, N]  \ \bigcap \ [b d \; [a, b]^{- \frac{\alpha_1 +1}{2}}, N],$$
where  $\bar x = x H$ for $x \in N$.  Thus $\K(N) = \gamma_2(N)$, and the proof is complete.   \hfill $\Box$

\end{proof}

We are now ready to prove Theorem A.

\noindent {\it Proof of Theorem A.} Let $G$ be a finite $p$-groups such that $\gamma_2(G)$ is of order $p^4$ and exponent $p$. Also let $\Z(G) \le \gamma_2(G)$. Notice that the nilpotency class of $G$ is at most $5$. By Remark \ref{remark} we have  $b(G) \ge 3$. Notice that $|G| \ge p^6$.  If $b(G) = 4$, then $\K(G) = \gamma_2(G)$. So we assume that $b(G) = 3$. If the nilpotency class of $G$ is at most $3$, then $\gamma_2(G)$ is abelian. When the nilpotency class of $G$ is $2$,  the assertion follows  from Lemmas \ref{cl2lemma1} - \ref{cl2lemma2}.  Now let the nilpotency class of $G$ be $3$. It follows from \cite{rJ80} that there is no group of order $p^6$ and nilpotency class $3$ satisfying the given hypothesis. Hence $|G| \ge p^7$.  If  $|\Z(G)| \le p^2$, then $\K(G) = \gamma_2(G)$ by Lemma \ref{p7lem2}. If $|G| \ge p^8$ and $|\Z(G)| = p^3$, then by Lemma \ref{main1} we have $\K(G) = \gamma_2(G)$. Finally if $|G| = p^7$ and $|\Z(G)| = p^3$, then by Lemma \ref{p7lem2} we have $\K(G) \ne \gamma_2(G)$. It only remains to handle the cases when the nilpotency class of $G$ is $4$ or $5$.

Let the nilpotency class of $G$ be $4$ and $b(G) = 3$.  If $|G| = p^6$, then it follows from Lemma \ref{p6lem} that  $\K(G) = \gamma_2(G)$ if and only if $|\Z(G)| = p$. So assume that $|G| \ge p^7$. By Lemma \ref{prelemma2} we have $|\Z(G)| \le p^2$. It follows from Theorem \ref{psthm} that $G$ admits a normal subgroup $H$ of order $p$ such that $|G/H:\Z(G/H)| = p^3$. The only choice for $H$, in this case, is $\gamma_4(G)$. If not, then $G/H$ will be of nilpotency class $4$, which is not possible as $(G/H)/\Z(G/H)$ can have nilpotency class at most $2$.  Thus $G/H$ has nilpotency class $3$. For the existence of such a group $G$, it is necessary that $(\gamma_2(G)/H ) \cap \Z(G/H)$ is of order $p^2$. Hence there exists a $2$-generator subgroup $S$ of $G$ such that $\gamma_2(G) = \gamma_2(S)$. Since $\gamma_2(S)/\gamma_3(S)$ is cyclic and $[\gamma_2(S), \gamma_3(S)] = 1$, it follows that $\gamma_2(S)$ is abelian.  Hence $\gamma_2(G)$ is abelian. Now invoking Theorem \ref{keyresult}, there exists a $2$-generator subgroup $T$ of $G$ such that $|T| = p^6$, $\gamma_2(T) = \gamma_2(G)$ and  $G$ is an amalgamated semidirect product of $T$ and a subgroup $K$ with $|\gamma_2(K)| \le p$.  Also $K = \gen{x_1, \ldots, x_k}$, for some $k \ge 1$ such that $[x_i, L] = H$ for $1 \le i \le k$. If $[x_i, T] = H$ for some $1 \le i \le k$, then the subgroup $M:=\gen{a, b, c}$, where $c = x_i$ is of order $p^7$ such that $\gamma_2(M) = \gamma_2(G)$. Hence it follows from  Lemma \ref{p7lem3} that $\K(G) = \gamma_2(G)$. Now assume that $[x_i, T] = 1$ for all $1 \le i \le k$.  Since $\Z(G) \le \gamma_2(G)$, it follows that $K$ is non-abelian. Thus $G$ is a central product of $T$ and $K$ amalgamating a subgroup containing $H$. If $|\Z(G)| = p$, then $|\Z(T)| = p$, and therefore by Lemma \ref{p6lem} we have $\K(T) = \gamma_2(T)$. Hence $\K(G) = \gamma_2(G)$.  Now we take the case when $|\Z(G)| = p^2$. 
Notice that $\gamma_2(K) = H$, and therefore there exists $x_i, x_j \in K$ such that $H = \gen{[x_i, x_j]}$ for some $1 \le i < j \le k$. Let $N:= \gen{a, b, c, d}$, where $c= x_i$ and $d = x_j$. Then $\gamma_2(G) = \gamma_2(N)$. Set $\bar N = N/H$.  As observed in the proof of Lemma \ref{p7lem3}, 
$$\gamma_2(\bar{N}) = \gen{[\bar a, \bar b], [\bar a, [\bar a, \bar b]],  [\bar b, [\bar a, \bar b]]}.$$
It is now not difficult to see that for $i, j \in \mathbb{F}_p$ we have
$$\gamma_2(\bar{N}) =   \bigcup \limits_{{i},{j}} \ [{\bar{a}}^i
{\bar{b}}^j \bar c, \bar N].$$
An easy computation also shows that $H \subseteq  [a^i b^j  c,  N]$ for all $i,
j \in \mathbb{F}_p$. Hence $\K(N)=\gamma_2(N)$ by Lemma \ref{prelemma3},
and therefore we have $\K(G)=\gamma_2(G)$.

Finally let  the nilpotency class of $G$ be $5$. We claim that  $b(G) = 4$, and hence we are done. Contrarily assume that $b(G) = 3$. Since the nilpotency class of $G/\Z(G)$ is $4$ and $\Z(G) \le \gamma_2(G)$, we have $|\Z(G)| = p$. Hence, in view of Theorem \ref{psthm}, the only choice for normal subgroup $H$ of $G$ such that  $|G/H:\Z(G/H)|=p^3$ is $\Z(G)$. Thus $|G/\Z(G) : \Z_2(G)/\Z(G)| = p^3$,  which implies that $|G/\Z_2(G)| = p^3$, where $\Z_2(G)$ denotes the second center of $G$.   This contradicts the fact that the nilpotency class of $G/\Z_2(G)$ is $3$. Our claim is now settled, and the proof of the theorem is complete.   \hfill $\Box$



\section{$2$-groups and proof of Theorem B}

We start with the following result of Wilkens \cite{bW07} on $2$-group of breadth $3$.

\begin{thm}\label{bettina}
Let $P$ be a finite $2$-group of breadth $3$. Then one of the following holds:
	
{\rm (i)} $|\gamma_2(P)| \le 2^3$.
	
{\rm (ii)} $|P: \Z(P)| \le 2^4$.
	
{\rm (iii)} $|\gamma_2(P)| = 2^4$ and there is $R$, $R\leq \Omega_1(\Z(P))$, $|R|= 2$, such that $|P/R : \Z(P/R)| \leq 2^3$. 
	
{\rm (iv)} $|\gamma_2(P)| = 2^4$ and $P$ is a central product $G \C_P(G)$, with $\C_P(G)$ is abelian and $G$ lies in  one of the following five classes of groups: 

{\rm (1)} There are $i,j \in \mathbb{N}$ with $G\cong \hat{G} / \gen{x^{4i}, y^{4j}}$, where $\gamma_4(\hat G) = 1 =   \mho_2(\gamma_2(\hat{G})) = \mho_1(\gamma_3(\hat{G}))$, and $\hat{G}$ is free in the category of these groups.

{\rm (2)} There are $i,j \in \mathbb{N}$ with $G\cong \hat{G} / \gen{x^{4i}, y^{2j}}$, where 
$$\hat{G}= \gen{x,y \mid [x,y]^y = [y,x], [y, {}_2x]^2=1=[y,{}_3x]^2=[y, {}_2x, y]= [y, {}_3x, y] =[y, {}_4x]}.$$

{\rm (3)} There are $i,j, k \in \mathbb{N}$ with $G\cong \hat{G} / \gen{x^{4i}, y^{2j}, z^{2k}}$, $\hat{G}= \gen{x,y,z}$ has $\gamma_4(G)=1=\mho_1(\gamma_3(G))$, and,  apart from that, is defined by the relations $$ [x, y]^4=1=[x,y]^2[x,y,y]=[x,y,z]=[x,z]^2=[x,z,z]=[x,z,y]=[x,z,x]=[y,z].$$

{\rm (4)} There are $i,j,k,l \in \mathbb{N}$ with $G\cong \hat{G} / \gen{x^{2i}, y^{2j}, a^{2k}, t^{2l}}$, $\hat{G}=\gen{a,t,x,y}$ of nilpotency class $3$ with additional relations $[x,a]^2=1=[x,a,w]=[y,t,w]=[y,t]^2$, $w \in \{a,t,x,y \}$, $[x,t]=[y,t][y,a]=1$, $[x,y]^4=1=[x,y,a]=[x,y,t]=[x,y,y][x,y]^2=[x,y,x][x,y]^2$, $[t,a] \in \gen{[x,y]^2}$.

{\rm (5)} There are $i,j,k,l,m \in \mathbb{N}$ such that $G= \hat{G}/{\gen{x^{2i}, v^{2j}, {v_1}^{2k}, {v_2}^{2l}, {v_3}^{2m} }}$, and $\hat{G}= \gen{a,v, v_1, v_2, v_3}$ is of nilpotency class $2$ with $\Phi(\hat{G}) \leq \Z(\hat{G})$ and is otherwise defined by $[v_2, x]=1=[v_1,v]=[v_3,x][v_3, v]$, $[v_i, v_j] \leq \gen{[v_3,x]}$. 
\end{thm}

The following result reduces our study to $2$-groups of nilpotency class $2$ and $3$, and to groups of order $2^7$ if the nilpotency class is $3$.
\begin{thm}\label{keyresult1}
Let $P$ be a finite $2$-group  such that   $\Z(P) \le \gamma_2(P)$, $\gamma_2(P)$ is elementary abelian of order $2^4$ and $b(P) = 3$.  Then 

{\rm (i)}  The nilpotency class of $P$ is either $2$ or $3$. 

{\rm (ii)} If the nilpotency class of $P$ is $2$, then $|P| \ge 2^8$.

{\rm (iii)} If the nilpotency class of $P$ is $3$, then there exists a $3$-generator subgroup $G$ of $P$ having the same nilpotency class as that of $P$ such that $\gamma_2(G) = \gamma_2(P)$. Moreover, $|G| = 2^7$,  and if $|P| \ge 2^8$, then $P$ is an amalgamated semidirect product of $G$ and a subgroup $K$ with $|\gamma_2(K)| = 2$. If $K$ is non-abelian, then it is isoclinic to  an extraspecial $2$-group.
\end{thm} 
\begin{proof}
For the given group $P$, one of the assertions (ii) - (iv) of  Theorem \ref{bettina} holds true. We start by noting that  $|P| \ge 2^7$. Also, if  Theorem \ref{bettina}(iv) holds, then  $P$ is isomorphic to a group in class (5) of  Theorem  \ref{bettina}(iv), which consists of groups of nilpotency class $2$.  So we only need to take into consideration the assertions (ii) - (iii) of  Theorem \ref{bettina}. Let the nilpotency class of $P$ be at least $4$. Then, by Lemma \ref{prelemma2}, $\Z(P)$ can not be maximal in $\gamma_2(P)$.  So in the case when Theorem \ref{bettina}(ii) holds, we must get $|\gamma_2(P)/\Z(P)| \ge 4$. Since $P$ is non-abelian, it follows that $|P| \le 2^6$, which is not possible as observed above. Now assume that Theorem \ref{bettina}(iii) holds. Then there exists a central subgroup $R$ of order $2$ such that $|P/R : \Z(P/R)| \le 8$. Also the nilpotency class of $P/R$ is at least $3$. This is possible only when  $|\gamma_2(P/R) : \gamma_2(P/R) \cap \Z(P/R)| = 2$, which, by Lemma \ref{prelemma2}, implies that the nilpotency class of $P/R$ is $3$. Thus  $(P/R)/\Z(P/R)$ is non-abelian, which is not possible as  shown in the next para.  Hence the nilpotency class of $P$ is either $2$ or $3$. Let the nilpotency class of $P$ be $2$ and $ |P| = 2^7$.  Then $P$ is generated by at most three elements, which is not possible. Hence $|P| \ge 2^8$ in this case.

Now we assume that the nilpotency class of $P$ is $3$. If Theorem \ref{bettina}(ii) holds, then, by the given hypotheses, it follows that $|\gamma_2(P)/\Z(P)| = 2$. Hence $P$ itself is a $3$-generator group of order $2^7$. Next assume that Theorem \ref{bettina}(iii) holds. Thus there exists a central subgroup $R$ of order $2$ such that $|P/R : \Z(P/R)| \le 8$. In our case, it is easy to deduce that $|P/R : \Z(P/R)| = 8$. Set $\bar{P} = P/R$.  We claim that $\bar{P}/\Z(\bar{P})$ is abelian. Contrarily assume that $\bar{P}/\Z(\bar{P})$ is non-abelian. Then the nilpotency class of $\bar{P}$ is $3$ and $\bar{P} = \gen{\bar a, \bar b, \Z(\bar{P})}$ for some $a, b \in P$, where $\bar x$ denotes $xR$ for $x \in P$.  Since the exponent of $\bar{P}/\Z(\bar{P})$ can not be $2$, it follows that $[\bar a, \bar b] = (\bar a^{\epsilon_1} \bar b^{\epsilon_2})^2$ modulo $\Z(\bar{P})$, where $\epsilon_i \in \mathbb{F}_2$. Hence $\C_{\bar{P}}([\bar a, \bar b])$ is maximal in $\bar{P}$.  This implies that $[[\bar a, \bar b], \bar{P}]$ is of order $2$, which contradicts the fact that $|\gamma_2(\bar{P})| = 8$. The claim is now settled. Thus the nilpotency class of $\bar{P}$ is $2$. If the exponent of $\bar{P}/\Z(\bar{P})$ is $4$, then $\bar{P} = \gen{\bar a, \bar b, \Z(\bar{P})}$ for some $a, b \in P$. This implies that $|\gamma_2(\bar{P})| = 2$, which again contradicts the fact that $|\gamma_2(\bar{P})| = 8$. Hence the exponent of $\bar{P}/\Z(\bar{P})$ is $2$, and therefore  $\bar{P} = \gen{\bar a, \bar b, \bar c , \Z(\bar{P})}$ for some $a, b, c \in P$. 

Let $G:=\gen{a, b, c}$. As proved in the reduction theorem for $p$ odd case,  $\gamma_2(G) = \gamma_2(P)$, and therefore  $|G| = 2^7$. Let $|P| \ge 2^8$. Then $P = \gen{a, b, c, x_1, \ldots, x_k}$ for some integer $k \ge 1$. Let $K := \gen{x_1, \ldots, x_k}$. Then $[x_i, K] = R$. It now follows that  $G$ and $K$ are the desired subgroups, which completes the proof.   \hfill $\Box$

\end{proof}

We deduce the following result using GAP \cite{GAP}. A theoretical proof goes on the lines of Lemma \ref{p7lem1}.

\begin{lemma}\label{cl3lemma}
Let $G$ be a group of order $2^7$  with $b(G)=3$, $\Z(G) \le \gamma_2(G)$  and $\gamma_2(G)$ elementary abelian of order $2^4$. Then $\K(G)=\gamma_2(G)$.
\end{lemma}

We again use GAP to establish the following two lemmas, whose theoretical proofs can be written on the lines of corresponding results of Section \ref{sec3}.

\begin{lemma} \label{28cl2lemma1} 
Let $G$ be a finite $p$-group of order $2^8$ and nilpotency class $2$  such that $\gamma_2(G)$ is elementary abelian of order $2^4$. If $G$ is not of conjugate type $\{1, 2^3\}$ or $\{1, 2^2, 2^3\}$, then $\K(G) \neq \gamma_2(G)$. Moreover, every element of $\gamma_2(G)$ can be written as a product of at most two elements from $\K(G)$.
\end{lemma}

\begin{lemma}\label{28cl2lemma2} 
Let $G$ be a finite $2$-group of order $2^8$,  nilpotency class $2$  and conjugate type $\{1,  2^3\}$ such that $\gamma_2(G)$ is elementary abelian of order $2^4$. Then $\K(G) = \gamma_2(G)$.
\end{lemma}


The proof of the following lemma goes on the lines of the proof of Lemma \ref{cl2lemma1a}. Since the proof for $p=2$ is a bit different towards the end,  we have decided to include a part of the proof here too. 

\begin{lemma}  \label{28cl2lemma3}
Let $G$ be a finite $2$-group of order $2^8$,  nilpotency class $2$  and conjugate type $\{1, 2^2, 2^3\}$ such that $\gamma_2(G)$ is elementary abelian of order $2^4$. Then $\K(G) = \gamma_2(G)$ if and only if $G$ admits a generating set $\{a, b, c, d\}$ such that $[a, b] = 1 = [c, d]$. Moreover, if $\K(G) \ne \gamma_2(G)$, then every element of $\gamma_2(G)$ can be written as a product of at most two elements from $\K(G)$.
\end{lemma}
\begin{proof}
The `if'  part is the same as that of Lemma \ref{cl2lemma1a}. For  `only if' part, as is also done in Lemma \ref{cl2lemma1a}, we  provide a contrapositive proof. We assume that $G$ admits no generating set $\{x_1, x_2, x_3, x_4\}$ such that  $[x_1, x_2] = 1 = [x_3, x_4]$.  By the given hypothesis, we can always choose a generating set $\{a, b, c, d\}$ for $G$ such that  $[a, c] = 1$ and none of the other basic commutators of weight two in generators is trivial. So we can assume that
\begin{equation*}
[c, d]= [a, b]^{t_1} [b, c]^{t_2} [b, d]^{t_3} [a,d]^{t_4},
\end{equation*}
where $t_i \in \mathbb{F}_2$ and $1 \le i \le 4$, and therefore $\gamma_2(G)= \langle [a, b], [b, c], [b, d], [a,d] \rangle$. 
		
We can write the preceding equation as
\begin{equation}\label{2cls2eqn1}
[b, a^{-t_1} c^{t_2} d^{t_3}][d, a^{-t_4}c]= 1.
\end{equation}
We claim that $t_3 = 0$.  If $t_3 = 1$, then replacing $d$ by $d' := a^{-t_1} c^{t_2} d$ and a proper substitution reduces \eqref{2cls2eqn1}  to $[d', b^{-1} a^{-t_4 } c]=1$. Now replacing $b$ by  $b' := b^{-1} a^{-t_4 } c$, we get a  generating set $\{ a, b',c,d' \}$ for $G$ such that $\gamma_2(G)= \langle [a, b'], [b', c], [b', d'], [c,d'] \rangle$ and $[a, c]=1=[b', d']$, which contradicts our hypothesis. Our claim is now settled.
		
So now onwards we assume that  $t_3=0$.  Hence  \eqref{2cls2eqn1} reduces to $[b, a^{-t_1} c^{t_2} ][d, c a^{-t_4}]= 1$. Now replace $c$ by $c' := c a^{-t_4} $. A simple computation gives
$$[b, a^{-t_1} c^{t_2} ] = [b, a^{-t_1} (c'a^{t_4})^{t_2} ]= [b, a^{t_4 t_2 - t_1} c'^{t_2} ].$$ 
Thus \eqref{2cls2eqn1} reduces to 
\begin{equation}\label{2cls2eqn1a}
[b, a^{t_4t_2 - t_1} c'^{t_2}][d, c'] = 1.
\end{equation}
		
We claim that $t_4 t_2 - t_1 \ne 0$.  Contrarily assume that $t_4 t_2 - t_1 = 0$. Then \eqref{2cls2eqn1a} takes the form $[b,  c'^{t_2}][d, c'] = 1,$ which gives $[b^{t_2}d, c']  = 1$. Now replacing $d $ by $d' := b^{t_2}d$, we get a generating set $\{a, b, c', d' \}$ of $G$ such that $[a, c'] = 1 = [d', c']$,  which gives that the size of the conjugacy class of $c'$ in $G$ is $2$,  a contradiction to the given hypothesis. This settles our claim.
		
So we now assume  $t_4 t_2- t_1 \neq 0$. Then replacing $a$ by $a' := a^{t_4 t_2 - t_1} c'^{t_2}$, we get a new generating set $\{ a', b, c', d \}$ such that $\gamma_2(G)= \langle [a', b], [b, c'], [b, d], [a',d] \rangle$,  $[a', c']=1$ and, by \eqref{2cls2eqn1a}, $[a', b] = [c', d]$.  Now we claim that $[a', b] [b, c'] [a', d] \notin  \K(G)$. Contrarily assume that $[a', b][b, c'] [a', d] \in \K(G)$. Thus
$$[a', b] [b, c'] [a', d] = [a'^{\alpha_1} b^{\alpha_2} c'^{\alpha_3} d^{\alpha_4 } , a'^{\beta_1} b^{\beta_2} c'^{\beta_3} d^{\beta_4}],$$
for some  $\alpha_i,  \beta_j \in \mathbb{F}_2$, $1 \le i, j \le 4$. Expanding the right hand side and  comparing powers both side, we get
\begin{eqnarray}
\alpha_1 \beta_4 + \alpha_4 \beta_1 &=& 1, \label{2eqn13}\\
\alpha_2 \beta_4  + \alpha_4 \beta_2 &=& 0, \label{2eqn14}\\
\alpha_2 \beta_3  + \alpha_3 \beta_2   &=& 1, \label{2eqn15}\\
\alpha_1 \beta_2  + \alpha_2 \beta_1 + \alpha_3 \beta_4  + \alpha_4
\beta_3 & = & 1. \label{2eqn16} 
\end{eqnarray}
		
First assume that $\alpha_2=0$. Then from \eqref{2eqn15} we get $\alpha_3= \beta_2 =1$, and hence from \eqref{2eqn14} we get $\alpha_4 = 0$, which using\eqref{2eqn13} gives $\alpha_1 = \beta_4 = 1$. But these values contradict \eqref{2eqn16}. So assume that $\alpha_2 = 1$. If $\beta_2 = 0$, then \eqref{2eqn14} gives $\beta_4 = 0$ and \eqref{2eqn15} gives $\beta_3 = 1$. Substituting $\beta_4 = 0$ in \eqref{2eqn13} we get $\alpha_4= \beta_1= 1$, which contradict \eqref{2eqn16}.  Hence $\beta_2 \neq 0$. 
		
Finally assume that $\alpha_2= \beta_2= 1$. Therefore above equations reduces to
\begin{eqnarray}
\alpha_1 \beta_4 + \alpha_4 \beta_1 &=& 1, \label{2eqn17}\\
 \beta_4  + \alpha_4  &=& 0, \label{2eqn18}\\
 \beta_3  + \alpha_3    &=& 1, \label{2eqn19}\\
\alpha_1   +  \beta_1 + \alpha_3 \beta_4  + \alpha_4
\beta_3 & = & 1. \label{2eqn20} 
\end{eqnarray}
		 
As, by \eqref{2eqn18}, $\alpha_4= \beta_4$, from \eqref{2eqn17} we get $\alpha_4 = 1 = \alpha_1  + \beta_1 $. Now putting $\alpha_4= \beta_4 = 1$ in \eqref{2eqn20} and  using \eqref{2eqn19}, we get $\alpha_1  + \beta_1 = 0$, which is not possible. Hence $[a', b] [b, c'] [a', d] \notin \K(G)$. The final assertion follows the similar way as in Lemma \ref{cl2lemma1a}. This completes the proof.	 \hfill $\Box$

\end{proof}

The proof of the following lemma goes on the lines of the proof of Lemma \ref{cl2lemma2}.
\begin{lemma} \label{29cl2lemma2} 
Let $G$ be a  finite $2$-group of nilpotency class $2$ and order at least  $2^9$   such that $\Z(G) = \gamma_2(G)$, Theorem \ref{bettina}(iii) holds and $\gamma_2(G)$ is elementary abelian of order $2^4$. Then $\K(G)=\gamma_2(G)$. 
\end{lemma}

Let $P$ be a $2$-group of breadth $3$ which satisfies Theorem \ref{bettina}(iv) and $\gamma_2(P)$ be elementary abelian of order $16$. Then, by a careful inspection,  it follows from Theorem \ref{bettina}  that $P$ is isoclinic to the group $T$  presented as
\begin{eqnarray}\label{lasteq}
\;\;\;\;T & =  & \langle  v_1, v_2, v_3, v_4, v_5 \mid  [v_4, v_2] = [v_5, v_1] =  1, [v_1, v_2] =  [v_3, v_4]^r, [v_2, v_3] =  [v_3, v_4]^s,\\
& &  [v_3, v_1] = [v_3, v_4]^t, [v_3, v_4]=[v_3, v_5], v_i^{2k_i} = 1, [v_i^2, v_j] =1  \;( 1\le i, j \le 5),\nonumber \\
& & [x, y, z] = 1 \mbox{ for all } x, y, z \in \{v_1, \ldots, v_5\} \rangle, \nonumber
\end{eqnarray}
for some positive integers  $k_i$'s  and  and some $r, s, t \in \mathbb{F}_2$. 

\begin{lemma}\label{lastlemma}
Let $G$ be a  finite $2$-group of breadth $3$ such that  Theorem \ref{bettina}(iv) holds, $\Z(G) \le \gamma_2(G)$ and $\gamma_2(G)$ is elementary abelian of order $2^4$. Then  $G$ is isoclinic to  the group $T$ given by \eqref{lasteq} for some $r, s, t \in \mathbb{F}_2$ and $k_i = 1$, $1 \le  i \le 5$, and  $\K(G) \neq \gamma_2(G) $ if and only if $r=s=t=0$.
\end{lemma}
\begin{proof}
It is not difficult to see that $G$ is isoclinic to  the group $T$ given by \eqref{lasteq} for $k_i = 1$, $1 \le  i \le 5$ and some $r, s, t \in \mathbb{F}_2$. We now prove the second assertion.
First assume that $r=s=t=0$. We claim that $[v_4, v_1] [v_4,v_3] [v_5, v_2] \notin \K(G)$. Otherwise, there exist $\alpha_i, \beta_i \in \mathbb{F}_2$ such that $$  [v_4, v_1] [v_4,v_3] [v_5, v_2]= [v_4^{\alpha_1} v_5^{\alpha_2} {v_1}^{\alpha_3} {v_2}^{\alpha_4} {v_3}^{\alpha_5} ,  v_4^{\beta_1} v_5^{\beta_2} {v_1}^{\beta_3} {v_2}^{\beta_4} {v_3}^{\beta_5}  ].$$
Expanding right side and comparing the powers of commutators, we get
\begin{eqnarray}
\alpha_1 \beta_2 + \alpha_2 \beta_1 &=& 0, \label{2eqn1}\\
\alpha_1 \beta_3  + \alpha_3 \beta_1 &=& 1, \label{2eqn2}\\
\alpha_2 \beta_4  + \alpha_4 \beta_2   &=& 1, \label{2eqn3}\\
\alpha_1 \beta_5  + \alpha_5 \beta_1 + \alpha_2 \beta_5  + \alpha_5
\beta_2 & = & 1. \label{2eqn4} 
\end{eqnarray}

First assume that $\alpha_1 = 0$. Then \eqref{2eqn2} gives $\alpha_3= \beta_1= 1$, hence by \eqref{2eqn1} we get $\alpha_2 = 0$, which gives $\alpha_4= \beta_2= 1$. But  these values contradict \eqref{2eqn4}. So now assume that $\alpha_1 = 1$. If $\beta_1 = 0$, then by \eqref{2eqn1} and \eqref{2eqn2} we get $\beta_2 = 0$ and $\beta_3 = 1$, respectively.  As $\beta_2 = 0$,  by \eqref{2eqn3} we get $\alpha_2= \beta_4= 1$;  but these values contradict \eqref{2eqn4}. Finally assume that $\alpha_1= \beta_1= 1$. Then by \eqref{2eqn1} and \eqref{2eqn3} we get $\alpha_2= \beta_2 = 1$, which again contradict \eqref{2eqn4}. Hence $[v_4, v_1] [v_4,v_3] [v_5, v_2] \notin \K(G)$. 

Conversely, assume that at least one of $r, s, t$ is non-zero. We'll show that  $\K(G) = \gamma_2(G)$. It is easy to see that except  $[v_4, v_1] [v_4,v_3] [v_5, v_2], \  [v_4, v] [v_4, v_1] [v_4,v_3] [v_5, v_2] $, all elements of $\gamma_2(G)$ lie in $\K(G)$. We'll first show that  $[v_4, v_1] [v_4,v_3] [v_5, v_2] \in \K(G)$.
If  $r=1$, then $[v_1, v_2] =  [v_3, v_4] = [v_4, v_3]$, and therefore 
$$ [v_4, v_1][v_5, v_2][v_4,v_3] =  [v_4, v_1][v_5, v_2][v_1,v_2] = [v_4 v_5 v_1, v_1 v_2 ].$$ 
So let $r=0$. If  $t=1$, then for any value of $s$ we have   $$ [v_4 v_5 v_1, v_1 v_2 v_3 ] = [v_4, v_1][v_4, v_3][v_5, v_2][v_5, v_3][v_1,v_3]=[v_4, v_1] [v_4, v_3] [v_5, v_2].$$
If $t=0$ and $s=1$, then $$ [v_4 v_5 v_2, v_1 v_2 v_3 ] = [v_4, v_1][v_4, v_3][v_5, v_2][v_5,v_3][v_2,v_3]=[v_4, v_1] [v_4,v_3] [v_5, v_2].$$

Now we take  $[v_4, v_5] [v_4, v_1] [v_4,v_3] [v_5, v_2]$. First let $r=1$. Then 
$$ [v_4 v_5 v_1, v_5 v_1 v_2]= [v_4, v_5] [v_4, v_1] [v_5, v_2][v_1, v_2].$$
Next let $r=0$. If $t=1$,   then $$ [ v_4 v_5 v_1, v_5 v_1 v_2 v_3 ]= [v_4, v_5] [v_4, v_1]  [v_4,v_3] [v_5, v_2] [v_5, v_3] [v_1, v_3]= [v_4, v_5] [v_4, v_1] [v_4,v_3] [v_5, v_2]$$
If $t = 0$ and $s=1$, then we finally get $$ [v_4 v_1  v_2, v_5 v_1 v_3 ]= [v_4, v_5] [v_4, v_1 ]  [v_4, v_3] [v_1, v_3] [ v_2, v_5]  = [v_4, v_5] [v_4, v_1] [v_4,v_3] [v_5, v_2].$$
Hence $[v_4, v_5] [v_4, v_1] [v_4,v_3] [v_5, v_2] \in \K(G)$, and the proof of the lemma is now complete.     \hfill $\Box$

\end{proof}

We can now write a proof of Theorem B.

\noindent {\it Proof of Theorem B.} Let $G$ be a finite $2$-groups such that $\gamma_2(G)$ is elementary abelian of order $16$. Also let $\Z(G) \le \gamma_2(G)$.  As in the case of odd primes,  we have  $b(G) \ge 3$ in this case too.   If $b(G) = 4$, then $\K(G) = \gamma_2(G)$. So we assume that $b(G) = 3$. Then it follows from Theorem \ref{keyresult1} that the nilpotency class of $G$ is either $2$ or $3$, and $|G| \ge 2^7$.  If the nilpotency class of $G$ is $2$, then the assertion follows  from Lemmas \ref{28cl2lemma1} - \ref{lastlemma}.  If the nilpotency class of $G$ is $3$,  then the assertion holds from Lemma \ref{cl3lemma}, and the proof is complete. \hfill $\Box$

\section{Examples}

In this section we present various types of examples of groups which occur in our study above for $p \ge 3$. Examples of $2$-groups are evident from GAP computations.

\vspace{.1in}

\noindent {\bf Groups of class $2$.} Let $F$ be the freest $p$-group of nilpotency class $2$ and exponent $p$ on $4$ generators, $a, b, c, d$ (say), where $p$ is an odd prime. Let $R:= \gen{[b,d], [a, d]}$. Then $G := F/R$ is a  group of nilpotency class $2$ and order $p^8$ such that $\K(G) \ne \gamma_2(G)$. If we take $R_1 := \gen{ [a,b][c,d],\; [a,c][b,d]^r}$, where $r$ is any fixed non-square integer modulo p. Then it follows from \cite[Theorem 1.2]{NY18} that $G_1:=F/R_1$ is a group of order $p^8$, nilpotency class $2$ and conjugate type $\{1, p^3\}$. For $p$-groups $G$, $p$ odd, of nilpotency class $2$ and order at least $p^9$, we know that $\K(G) = \gamma_2(G)$. Such examples of order $\ge p^{10}$ can be constructed by taking a central product of the group $G:=F/R$ and any finite extraspecial $p$-group $K$ amalgamating $\gen{[\bar c, \bar d]} = \gamma_2(K)$. Constructing such examples of order $p^9$ is also easy, as explained in the proof of Lemma \ref{cl2lemma2}.

\vspace{.1in}

\noindent {\bf Groups of class $3$.} We present five types of $p$-group of nilpotency class $3$ and order $p^7$, where $p$ is an odd prime. Consider the group presented as
\begin{eqnarray*}
 G &=& \Big\langle  \alpha_1,  \alpha_2, \alpha_3, \alpha_4, \alpha_5,
\alpha_6, \gamma \mid [\alpha_2,  \alpha_1]=\alpha_4, [\alpha_3, 
\alpha_1] =\alpha_5, [\alpha_3,  \alpha_2]=\alpha_6,\\
 & & \;\ [\alpha_4, \alpha_2]=\gamma, [\alpha_5,  \alpha_3]=\gamma, [\alpha_5,  \alpha_2]=\gamma,
[\alpha_6,  \alpha_1]=\gamma, {\alpha_1}^p=\gamma,\\
 & & \;\ \alpha{_i}^p=\gamma^p=1~ (2 \le i \le 6) \Big\rangle.
\end{eqnarray*}
Notice that  $|\Z(G)| = p$ and $\K(G) = \gamma_2(G)$.

The following group $G$ is such that  $|Z(G)|=p^2=|\gamma_3(G)|$ and $\K(G) = \gamma_2(G)$.
\begin{eqnarray*}
 G &=& \Big\langle  \alpha_1,  \alpha_2, \alpha_3, \alpha_4, \alpha_5,
\alpha_6, \gamma \mid [\alpha{_i},  \alpha_1]=\alpha{_{i+2}}, [\alpha_5, 
\alpha_1]=\gamma,\\
 & & \;\ \alpha{_1}^p=\alpha{_i}^p=\alpha{_5}^p=\alpha{_6}^p=\gamma^p=1~
(2 \le i \le 4) \Big\rangle.
\end{eqnarray*}

The next group $G$ is such that  $|Z(G)|=p^2$, $|\gamma_3(G)|=p$ and $\K(G) = \gamma_2(G)$.
\begin{eqnarray*}
 G &=& \Big\langle  \alpha_1,  \alpha_2, \alpha_3, \alpha_4, \alpha_5,
\alpha_6, \gamma \mid [\alpha_2,  \alpha_1]=\alpha_4, [\alpha_3, 
\alpha_1]=\alpha_5, [\alpha_3,  \alpha_2]=\alpha_6,\\
 & & \;\ [\alpha_4, \alpha_2] = \gamma, [\alpha_5,  \alpha_3]=\gamma,  {\alpha_1}^p=\gamma,
\alpha{_i}^p=\gamma^p=1~ (2 \le i \le 6) \Big\rangle.
\end{eqnarray*}

We now present a group $G$ such that $|Z(G)|=p^3$, $|\gamma_3(G)|=p$ and $\K(G) \ne \gamma_2(G)$.
\begin{eqnarray*}
 G &=& \Big\langle  \alpha_1,  \alpha_2, \alpha_3, \alpha_4, \alpha_5,
\alpha_6, \gamma \mid [\alpha_2,  \alpha_1]=\alpha_4, [\alpha_3, 
\alpha_1]=\alpha_5, \\
 & & \;\  [\alpha_3,  \alpha_2]=\alpha_6, [\alpha_4, \alpha_1]=\gamma,  \alpha{_i}^p=\gamma^p=1~ (1 \le i \le6) \Big\rangle.
\end{eqnarray*}

 Finally we present a group $G$ such that $|Z(G)|=p^3$, $|\gamma_3(G)|=p^2$ and $\K(G) \ne \gamma_2(G)$.
\begin{eqnarray*}
 G &=& \Big\langle  \alpha_1,  \alpha_2, \alpha_3, \alpha_4, \alpha_5,
\alpha_6, \gamma \mid [\alpha_2,  \alpha_1]=\alpha_4, [\alpha_3, 
\alpha_1]=\alpha_5, \\
 & & \;\   [\alpha_4,  \alpha_1]=\alpha_6, [\alpha_4, \alpha_2] =\gamma, \alpha{_i}^p=\gamma^p=1~ (1 \le i \le 6) \Big\rangle.
\end{eqnarray*}

\vspace{.1in}

\noindent {\bf Groups of class $4$.} We present two types of $p$-groups of nilpotency class $4$ and order $p^7$, where $p$ is an odd prime. Consider the group presented as
   \begin{eqnarray*}
 G &=& \Big\langle  \alpha_1,  \alpha_2, \alpha_3, \alpha_4, \alpha_5,
\alpha_6, \gamma \mid [\alpha_2,  \alpha_1]=\alpha_4, [\alpha_4, \alpha_1]=\alpha_5, [\alpha_4,  \alpha_2]=\alpha_6, \\
 & & \;\ [\alpha_4,  \alpha_3]= [\alpha_5,  \alpha_1]= [\alpha_6,  \alpha_2]=\gamma,  [\alpha_3, 
\alpha_1]=\gamma, \alpha{_i}^p=\gamma^p=1~ (2 \le i \le 6) \Big\rangle.
\end{eqnarray*}
For this group  $|Z(G)|=p$ and $\K(G) = \gamma_2(G)$.

The following is a group $G$ such that $|Z(G)|=p^2$ and $\K(G) = \gamma_2(G)$.
   \begin{eqnarray*}
 G &=& \Big\langle  \alpha_1,  \alpha_2, \alpha_3, \alpha_4, \alpha_5,
\alpha_6, \gamma \mid [\alpha_2,  \alpha_1]=\alpha_4, [\alpha_4, 
\alpha_1]=\alpha_5, [\alpha_4,  \alpha_2]=\alpha_6, \\
 & & \;\ [\alpha_3, \alpha_2] = [\alpha_5,  \alpha_1]= \gamma,  {\alpha_1}^p=\gamma,
\alpha{_i}^p=\gamma^p=1~ (2 \le i \le 6) \Big\rangle.
\end{eqnarray*}

We conclude with the remark that there are total $159$ groups $G$ of order $5^7$ such that $\gamma_2(G)$ is elementary abelian of order $5^5$, and $\K(G) \neq \gamma_2(G)$ for $141$ such  groups.

\end{document}